\documentclass[11pt, letterpaper, notitlepage, oneside]{article}

\usepackage{palatino,url}
\usepackage{graphicx}
\usepackage{amsmath}
\usepackage{amsfonts}
\usepackage{mathrsfs}
\usepackage{amsthm}
\usepackage{amssymb}
\usepackage[all]{xy}
\usepackage{verbatim}
\usepackage{color}
\usepackage{xr} 

\usepackage[vcentermath]{youngtab}
\newcommand{\Y}[1]{{\tiny\yng(#1)}}

\newcommand{\w}[1][]{\omega_{#1}}
\newcommand{\s}{\sigma}
\newcommand{\p}[1][]{\pi_{#1}}
\newcommand{\y}{\lambda}

\newcommand{\g}{\gamma}
\renewcommand{\a}{\alpha}
\renewcommand{\b}{\beta}
\renewcommand{\d}{\delta}
\renewcommand{\p}{\rho}
\renewcommand{\t}{\tau}

\newcommand{\Ind}{\text{Ind}}

\newcommand{\Res}{\text{Res}}
\newcommand{\Hom}{\text{Hom}}

\newcommand{\Aut}{\text{Aut}}
\newcommand{\End}{\text{End}}
\newcommand{\Inn}{\text{Inn}}

\newcommand{\Span}{\text{Span}}
\renewcommand{\dim}{\text{dim}}

\newcommand{\codim}{\text{codim}}

\newcommand{\FIW}{\text{FI}_{\mathcal{W}}}
\newcommand{\oFIW}{\text{FI}_{\overline{\mathcal{W}}}}
\newcommand{\FI}{\text{FI}}

\newcommand{\PS}{P\Sigma}

\newcommand{\cA}{\mathscr{A}}

\newcommand{\cM}{\mathscr{M}}

\renewcommand{\S}{\Sigma}
\newcommand{\Gn}{{ G}^{*n}}
\newcommand{\bn}{{\bf n}}
\newcommand{\bm}{ {\bf m}}
\newcommand{\bno}{{\bf n_0}}
\newcommand{\bmo}{ {\bf m_0}}

\newcommand{\ba}{ {\bf a}}

\newcommand{\bp}{ {\bf p}}
\newcommand{\bpo}{ {\bf p_0}}
\newcommand{\bX}{ {\bf X}}
\newcommand{\bY}{ {\bf Y}}

\newcommand{\W}{\mathcal{W}}
\newcommand{\oW}{\overline{\mathcal{W}}}
\newcommand{\sh}{\sharp}
\newcommand{\be}{{\bf \textbf{e}}}

\newcommand{\C}{\mathbb{C}}
\newcommand{\N}{\mathbb{N}}
\newcommand{\Q}{\mathbb{Q}}
\newcommand{\R}{\mathbb{R}}
\newcommand{\Z}{\mathbb{Z}}

\newtheorem{thm}{Theorem}[section]
\newtheorem{prop}[thm]{Proposition}
\newtheorem{lem}[thm]{Lemma}
\newtheorem{cor}[thm]{Corollary}

\theoremstyle{definition}
\newtheorem{defn}[thm]{Definition}
\newtheorem{rem}[thm]{Remark}
\newtheorem{ex}[thm]{Example}
\newtheorem{prob}[thm]{Problem}

\theoremstyle{plain}

\date{\today}
\title{$\FIW$--modules and constraints on classical Weyl group characters}
\author{Jennifer C. H. Wilson}

\begin{document}
\maketitle

\begin{abstract}

In this paper we study the characters of sequences of representations of any of the three families of classical Weyl groups $\W_n$: the symmetric groups, the signed permutation groups (hyperoctahedral groups), or the even-signed permutation groups. Our results extend work of Church, Ellenberg, Farb, and Nagpal \cite{CEF}, \cite{CEFN} on the symmetric groups. We use the concept of an \emph{$\FIW$--module}, an algebraic object that encodes the data of a sequence of $\W_n$--representations with maps between them, defined in the author's recent work \cite{FIW1}. 

We show that if a sequence $\{V_n\}$ of $\W_n$--representations has the structure of a \emph{finitely generated $\FIW$--module}, then there are substantial constraints on the growth of the sequence and the structure of the characters:  for $n$ large, the dimension of $V_n$ is equal to a polynomial in $n$, and the characters of $V_n$ are given by a \emph{character polynomial} in signed-cycle-counting class functions, independent of $n$. We determine bounds the degrees of these polynomials.

We continue to develop the theory of $\FIW$--modules, and we apply this theory to obtain new results about a number of sequences associated to the classical Weyl groups: the cohomology of complements of classical Coxeter hyperplane arrangements, and the cohomology of the pure string motion groups (the groups of symmetric automorphisms of the free group). 
\end{abstract}

\setcounter{tocdepth}{2}
{\small \tableofcontents }

\section{Introduction}

In recent work \cite{FIW1} we developed the theory of \emph{$\FIW$--modules} to study sequences of representations of any of the three families of classical Weyl groups: the symmetric groups $S_n$, the hyperoctahedral groups (signed permutation groups) $B_n \cong (\Z/2\Z)^n \rtimes S_n$, and their index-two subgroups of even-signed permutation groups $D_n$. The results generalize work of Church, Ellenberg, Farb, and Nagpal \cite{CEFN, CEF} on sequences of $S_n$--representations. 

 Let $\W_n$ denote one of these families. Many naturally occurring sequences of $\W_n$--representations $\{V_n\}_n$ with maps $V_n \to V_{n+1}$ carry the structure of an \emph{finitely generated $\FIW$--module}, an elementary and often easily established condition which we define in detail below. 
 
 We analyze two such families of sequences in Section \ref{SectionApplications}. The first is the cohomology of the pure string motion groups $\PS_n$, groups related to the pure braid groups. The group $\PS_n$ can also be identified with the group of pure symmetric automorphisms of the free group $F_n$. The second family is the cohomology of the (complexified) complements $\cM_{\W}(n)$ of the hyperplanes fixed by reflections in the Coxeter groups $\W_n$. The hyperplane complements $\cM_{\W}(n)$ are  objects of classical and current mathematical interest.
 
In this paper we develop general results on the characters of finitely generated $\FIW$--modules over characteristic zero. We prove that these characters admit very specific descriptions: they are (for $n$ large) given by \emph{character polynomials}, polynomials in signed-cycle-counting class functions, as described in Section \ref{SectionCharPolys}. 

\newtheorem*{charpolyintro}{Theorem \ref{CharPolySummary}}
\begin{charpolyintro}{\bf (Finitely generated $\FIW$--modules have character polynomials).} Let $V$ be an $\FIW$--module over characteristic zero finitely generated in degree $\leq d$. Let $\chi_{V_n}$ denote the character of the $\W_n$--representation $V_n$. Then there exists a unique character polynomial $F_V$ of degree at most $d$ such that $F_V(\sigma) = \chi_{V_n}(\sigma)$ for all $\sigma \in \W_n$, for all $n$ sufficiently large. 
\end{charpolyintro}
\noindent Church--Ellenberg--Farb proved this result in type A \cite[Theorem 2.67]{CEF}. Theorem \ref{WnPersinomial} gives a more detailed statement of the result in type B/C and D, including bounds on the stable range.

These character polynomials provide a description of the characters of $V_n$ that is independent of $n$, and moreover the character polynomials' restrictive structure reflects the strong constraints on the $\W_n$--subrepresentations that can appear in a finitely generated $\FIW$--module. 

One immediate consequence of the existence of character polynomials is that the dimensions dim$(V_n)$ are (for $n$ large) equal to a polynomial in $n$. We prove in Theorem \ref{PolyDimCharP} that this same phenomenon holds for sequences of representations $\{V_n\}_n$ over arbitrary fields when $V_n$ admits a finitely generated $\FIW$--module structure. Since $\W_n$--representations over positive characteristic need not be completely reducible, the study of these sequences is more subtle -- but Theorem \ref{PolyDimCharP} nonetheless offers some control over their growth.

\newtheorem*{polygrowthcharp}{Theorem \ref{PolyDimCharP}}
\begin{polygrowthcharp}{\bf (Polynomial growth of dimension over arbitrary fields).} Let $k$ be any field, and let $V$ be a finitely generated $\FIW$--module over $k$. Then there exists an integer-valued polynomial $P(T) \in \Q[T]$ such that $$ \dim_k (V_n) = P(n) \qquad \text{for all $n$ sufficiently large.}$$
\end{polygrowthcharp}

\noindent Theorem \ref{PolyDimCharP} is proved in type A by Church--Ellenberg--Farb--Nagpal \cite[Theorem 1.2]{CEFN}.

These stability results for $\FIW$--modules extends our earlier work \cite{FIW1}. We proved that if an $\FIW$--module over characteristic zero is finitely generated,  the decompositions of the representations $V_n$ into irreducible subrepresentations are in some sense eventually constant in $n$: these sequences are \emph{uniformly representation stable} as defined by Church--Farb \cite{RepStability}. See Section \ref{SectionRepStabilityReview} for a complete definition.

In this paper we introduce $\FIW\sh$--modules, a subclass of $\FIW$--modules with additional symmetries that allow us to deduce even stronger restrictions on the structure of the underlying sequence of representations; see Section \ref{SectionFISharp}.

Examples of sequences of $\W_n$--representations with finitely generated $\FIW$--module and $\FIW\sh$--module structures are prevalent throughout the fields of geometry, topology, algebra, and combinatorics. We establish this structure in two examples: the cohomology of the pure string motion groups $\PS_n$ (Section \ref{SectionPureStringMotion}) and the cohomology of the hyperplane complements $\cM_{\W}(n)$ (Section \ref{SectionHyperplaneComplements}).  Theorem \ref{ApplicationsSummaryIntro} summarizes our results.

\begin{thm} \label{ApplicationsSummaryIntro} Let $\W_n$ denote $S_n$, $B_n$, or $D_n$. Let $X_n$ denote either the pure string motion groups $\PS_n$ or the (complexified) complements $\cM_{\W}(n)$ of the reflecting hyperplanes of $\W_n$. Then in each degree $m$ the sequence $\{ H^{m}(X_n; \Q) \}_n$ is an $\FIW$--module finitely generated  in degree $\leq 2m$. We show
\begin{itemize}
 \item The sequence $H^{m}(X_n; \Q)$ is uniformly representation stable, stabilizing for $n \geq 4m$.
 \item For all values of $n$, the characters of $H^{m}(X_n; \Q)$ are given by a unique character polynomial of degree at most $2m$.
\end{itemize} 
\end{thm}
\noindent These and other results for $\PS_n$ and $\cM_{\W}(n)$ are described in detail in Section \ref{SectionApplications}.

The theory of $\FIW$--modules derives much of its strength from the extra algebraic structure provided to these sequences of $\W_n$--representations. In this framework each sequence $\{ V_n\}_n$ is encoded as a single object, an $\FIW$--module, in a category that closely parallels the category of modules over a ring. In Section \ref{SectionFIWFoundationsReview} we review the structure of maps, presentations, direct sums, and tensor products of $\FIW$--modules. We proved that $\FIW$--modules over Noetherian rings are Noetherian \cite[Section \ref{FIW1-SectionNoetherian}]{FIW1}; see Section \ref{SectionNoetherianReview}.  We developed induction and restriction operations between the three families of groups \cite[Sections \ref{FIW1-Section:Restriction} and \ref{FIW1-Section:Induction}]{FIW1}. These operations are reviewed in Section \ref{SectionReviewRestrictionInduction}. 

In our previous work \cite[Section \ref{FIW1-SectionMurnaghan}]{FIW1} we gave an application analogous to Murnaghan's classical theorem for Kronecker coefficients \cite{MurnaghanKronecker}, a stability theorem concerning tensor products of $S_n$--representations. We showed that type B/C and D versions of Murnaghan's result follow readily from the $\FIW$--module theory \cite[Theorem \ref{FIW1-MurnaghanWn} and Corollary \ref{FIW1-MurnaghanDn}]{FIW1}. These stability results have simple interpretations in this $\FIW$--module context: tensor products of finitely generated $\FIW$--modules are themselves finitely generated. 

The theory of $\FIW$--modules gives a conceptual foundation and a language to describe stability phenomena in sequences of $\W_n$--representations such as these.

\subsection{$\FIW$--modules and finite generation}

\begin{defn}{\bf (The Category $\FIW$).} \label{DefnFIW} Here we will define the three categories $\FI_A \subseteq \FI_D \subseteq \FI_{BC}$, denoted generically by $\FIW$. Consider the category whose objects are $\varnothing$ and for each $n \in \N$ the finite set $\bn = \{1, -1, 2, -2, \ldots, n, -n \}$,  and whose morphisms are all injective maps. We define $\FIW$ to be the smallest subcategory containing the morphisms $\W_n \subseteq$ End$(\bn)$ and the canonical inclusion maps $I_n: \bn \hookrightarrow {\bf (n+1)}$.  
	
	In each case, the endomorphisms End$(\bn)$ of $\FIW$ are precisely the Weyl group $\W_n$. The category $\FI_A$ is equivalent to the category of all finite sets and injective maps, denoted FI by \cite{CEF}. It turns out that for $n \neq m$, the set of morphisms $\bm \to \bn$ in $\FI_D$ and $\FI_{BC}$ are the same; see \cite[Remark 3.1]{FIW1}.
	
	A description of each category is given in Table \ref{FICatTable}.

\quad \\
\begin{table}[h]
\noindent {\footnotesize  
\begin{tabular}{|p{1.2cm}|p{3.2cm}|p{9.7cm}|} 
\hline && \\  Category & Objects &  Morphisms \\ && \\ \hline && \\ 
 $\FI_{BC}$ & $\bn = \{ \pm 1, \pm 2, \ldots, \pm n \}$& $\{$ injections  $f:\bm \to \bn \text{  $\; | \; f(-a) = -f(a) \; \; \forall$ $a \in \bm$} \}$ \\ 
 & ${\bf 0} = \varnothing $ & \qquad $\End(\bn) \cong B_n$ \\ && \\
$\FI_D$ &$\bn = \{ \pm 1, \pm 2, \ldots, \pm n \}$ \newline ${\bf 0} = \varnothing$  &  $\{$ injections $f:\bm \to \bn \; |\; f(-a) = -f(a)$ $\; \; \forall$ $a \in \bm$;  \newline  isomorphisms must reverse an even number of signs $\}$ \\ 
&& \qquad $\End(\bn) \cong D_n$ \\  && \\
$\FI_A$  & $\bn = \{ \pm 1, \pm 2, \ldots, \pm n \}$ \newline ${\bf 0} = \varnothing $  &  $\{$ injections $f:\bm \to \bn \; | \;  f(-a) = -f(a)$ $\; \; \forall$ $a \in \bm$;   $f$ preserves signs$\}$ \indent $ \qquad \End(\bn) \cong S_n$ \\ && \\ \hline
\end{tabular} 
\caption{The Categories $\FIW$} \label{FICatTable} 
} \end{table}
\quad \\

\noindent For $m<n$ we denote by $I_{m,n}$ the canonical inclusion $\{\pm 1, \ldots, \pm m \} \hookrightarrow \{\pm 1, \ldots, \pm n \}$ and abbreviate $I_n := I_{n, (n+1)}$. 
\end{defn}

\begin{defn} {\bf ($\FIW$--module and $\FIW$--module maps).} \label{DefnFIWModuleFIModuleMap} An \emph{$\FIW$--module} $V$ over a commutative ring $k$ is a (covariant) functor $$V : \FIW \longrightarrow k\text{--Mod}$$ from $ \FIW$ to the category of $k$--modules. We denote $V_n := V(\bn)$ and $f_* := V(f)$. 
	
	$\FI_A$--modules are precisely the $\FI$--modules studied in \cite{CEF, CEFN}. 

A \emph{co--$\FIW$--module} over $k$ is a functor from the dual category $$\FIW^{\mathrm{op}} \longrightarrow k\text{--Mod}.$$

A \emph{map of $\FIW$--modules} $F: V \longrightarrow W$ is a natural transformation, that is, a sequence of maps $F_n: V_n \to W_n$ that commute with the action of the $\FIW$ morphisms. $\FIW$--module injections, quotients, kernels, cokernels, direct sums, etc, are defined pointwise. 
\end{defn}

\begin{defn} {\bf (Finite generation, Degree of generation).} \label{DefnFinGen}  Let $V$ be an $\FIW$--module. Given a subset $S \subseteq \coprod_{n=0}^{\infty} V_n$, the sub--$\FIW$--module \emph{generated by $S$} is the smallest sub--$\FIW$--module $U$ of $V$ containing $S$. $S$ is a \emph{generating set} for $U$: the images of these elements under the $\FIW$ morphisms span each $k[\W_n]$--module  $U_n$.

An $\FIW$--module $V$ is \emph{finitely generated} if it has a finite generating set, and $V$ is generated in \emph{degree $\leq d$} if it has a generating set contained in $\coprod_{n=0}^{d} V_n$. \end{defn} 

{\bf Examples of $\FIW$--modules. \qquad} To illustrate this concept we give some first examples of $\FIW$--modules over $\Q$. Fix the Weyl group family $\W_n$ to be $S_n$, $D_n$, or $B_n$. To specify the $\FIW$--module structure on a sequence $\{V_n\}$, it suffices to state the $\W_n$--actions and the maps $(I_n)_*: V_n \to V_{n+1}$. associated to the natural inclusions $I_n$. 

\begin{ex} \label{FIModuleExamples}The following are $\FIW$--modules.
\begin{enumerate}
\item \label{FIExample-TrivialReps} {\bf Example: Trivial representations. } For $n \geq 0$ let $V_n = \Q$ be the trivial $\W_n$--representation with isomorphisms $(I_n)_*: V_n \cong V_{n+1}$. These spaces form an $\FIW$--module with a single generator in degree $0$. 
\item \label{FIExample-SignedPermMatrices} {\bf Example: Signed permutation matrices.} The groups $D_n$ and $B_n$ are canonically represented by $n\times n$ \emph{signed permutation matrices}, that is, generalized permutation matrices with nonzero entries equal to $1$ or $-1$. Throughout Example \ref{FIModuleExamples} we let $\Q^n$ denote the representation of $S_n$, $D_n$, or $B_n$ by (signed) permutation matrices. The representations $V_0 = 0$ and $V_n=\Q^n$ with their natural inclusions form an $\FIW$--module finitely generated in degree $1$. 
\item \label{FIExample-Schur} {\bf Example: $j$-fold powers.} For any integer $j$, the $j$-fold tensor power, exterior power, and symmetric power on $\Q^n$ each form an $\FIW$--module finitely generated in degree $j$. More generally, composing any $\FIW$--module $V$ with another functor $k$--Mod $\to k$--Mod will yield a new $\FIW$--module.

\item \label{FIExample-Represented} {\bf Example: Represented functors $M_{\W}(\bm)$.} For fixed integer $m \geq 0$, the sequence of $k$-modules $$M_{\W}(\bm)_n := k \big[ \Hom_{\FIW}( \bm, \bn)\big]$$ form an $\FIW$-module, with $\FIW$ morphisms acting on basis elements $e_f, f \in \Hom_{\FIW}( \bm, \bn) $ by postcomposition. These are in a sense the ``free" $\FIW$--modules; see Definition \ref{Defn:M(m)Review}.
 
\item \label{FIExample-Induced} {\bf Example: The $\FIW$--modules $M_{\W}(U)$.} Fix an integer $d$ and a $\W_d$--representation $U$. Let $\Q$ denote the trivial $\W_{n-d}$--representation. Let $U \boxtimes \Q$ denote the $(\W_d \times \W_{n-d})$--representation given by the \emph{external tensor product} of $U$ and $\Q$. Define
\begin{align*}
M_{\W}(W)_n := &\left\{ \begin{array}{ll} 0,\; n<d & \\ 
					 \Ind_{\W_d \times \W_{n-d}}^{\W_n} U \boxtimes \Q,\; n \geq d &  \end{array} \right.
\end{align*}

Then there are induced maps $M_{\W}(U)_n \to M_{\W}(U)_{n+1}$ giving this sequence the structure of an $\FIW$--module finitely generated in degree $d$. Specifically, the natural inclusion of  $(\W_d \times \W_{n-d})$--representations
\begin{align*} U \boxtimes \Q &\hookrightarrow \Res^{\W_{n+1}}_{\W_d \times \W_{n-d}} \Ind_{\W_d \times \W_{n+1-d}}^{\W_{n+1}}  U \boxtimes \Q \\ &= \Res^{\W_{n}}_{\W_d \times \W_{n-d}} \Res^{\W_{n+1}}_{\W_n} \Ind_{\W_d \times \W_{n+1-d}}^{\W_{n+1}}  U \boxtimes \Q \end{align*} and the universal property of induction give $\W_n$-equivariant maps $$ \underbrace{\Ind_{\W_d \times \W_{n-d}}^{\W_n} U \boxtimes \Q}_{M_{\W}(U)_n} \longrightarrow \Res^{\W_{n+1}}_{\W_n} \underbrace{\Ind_{\W_d \times \W_{n+1-d}}^{\W_{n+1}}  U \boxtimes \Q}_{M_{\W}(U)_{n+1}}, $$ which define the $\FIW$--module structure on $M_{\W}(U)$. An alternate, equivalent description of $M_{\W}(U)$ is given in Definition \ref{Defn:M(m)Review}.
	
Taking $d=0$ and $U$ the trivial  representation $\Q$, we  recover Example \ref{FIModuleExamples}.\ref{FIExample-TrivialReps}. Taking $d=1$ and $U \cong \Q^1$ recovers Example \ref{FIModuleExamples}.\ref{FIExample-SignedPermMatrices}.  For any $d$, taking $U$ to be the regular representation $k[\W_d]$ recovers Example \ref{FIModuleExamples}.\ref{FIExample-Represented}. The $\FIW$--modules $M_{\W}(U)$ are discussed in Section \ref{SectionMW}.

\item \label{FIExample-ZeroMaps} {\bf Example: Zero maps. } Let $\{V_n\}$ be any sequence of non-zero rational $\W_n$--representations, and and let $(I_n)_*$ be the zero maps. These form an $\FIW$--module that is infinitely generated, with infinite degree of generation.  
\item \label{FIExample-Truncated} {\bf Example: Torsion and truncated $\FIW$--modules.} Define $\FIW$--modules $V$ and $U$ by
\begin{align*}
 V_n :=\left\{ \begin{array}{ll} \Q^n, \text{ with $(I_n)_*$ the natural inclusions},\; n<20 & \\ 
					 0, \; n \geq 20 &  \end{array} \right. 
\end{align*} \begin{align*}
&  U_n  :=\left\{ \begin{array}{ll}\; 0, \;   n<20 & \\ 
					  \Q^n, \text{ with $(I_n)_*$ the natural inclusions}, \; n \geq 20 &  \end{array} \right.
\end{align*}
Then the ``torsion" $\FIW$-module $V$ is finitely generated in degree $1$, and the ``truncated" $\FIW$-module $U$ is finitely generated in degree $20$. 
\end{enumerate}
\end{ex}

In contrast to Example \ref{FIModuleExamples}, the sequence of alternating representations and the sequence of regular representations $\Q[\W_n]$, each with their canonical inclusions, do \emph{not} form $\FIW$--modules; see  \cite[Examples \ref{FIW1-NonExampleRegularReps} and \ref{FIW1-NonExampleAltReps}]{FIW1}.


\subsection{Character polynomials in type B/C and D} Let $k$ be a field of characteristic zero. One of our main results is that the sequence of characters of a finitely generated $\FIW$--module over $k$ is, for $n$ large, equal to a \emph{character polynomial} which does not depend on $n$. This was proven for symmetric groups in \cite[Theorem 2.67]{CEF}, and here we extend these results to the groups $D_n$ and $B_n$.

Character polynomials for the symmetric groups date back to Murnaghan \cite{MurnaghanCharacterPolynomials} and Specht \cite{SpechtCharaktere}; they are described in Macdonald \cite[I.7.14]{MacdonaldSymmetric}. In Section \ref{SectionCharPolys} we introduce character polynomials for the groups $B_n$ and $D_n$, in two families of \emph{signed} variables. We use the classical results for $S_n$ to derive formulas for the character polynomials of irreducible $B_n$--representations (Theorem \ref{ThmWnIndependentCharacters}), and use these formulas to study the characters of $\FIW$--modules in type B/C and D. 

Conjugacy classes of the hyperoctahedral group are classified by \emph{signed cycle type}, see Section \ref{RepTheoryBnReview} for a description. We define the functions $X_r$, $Y_r$ on $\coprod_{n=0}^{\infty} B_n$ such that 
\begin{align*}
 X_r (\w) & \text{ is the number of positive $r$--cycles in $\w$,} \\
Y_r (\w) & \text{ is the number of negative $r$--cycles in $\w$.}
\end{align*}
The functions $X_r, Y_r$ are algebraically independent as class functions on  $\coprod_{n=0}^{\infty} B_n$, and so they form a polynomial ring $k[X_1, Y_1, X_2, Y_2, \ldots]$ whose elements span the class functions on $B_n$ for each $n \geq 0$.

We prove that the sequence of characters of $\{V_n\}$ associated to any finitely generated $\FI_{BC}$--module or $\FI_{D}$--module $V$ over a field of characteristic zero are equal to a unique element of $k[X_1, Y_1, X_2, Y_2, \ldots]$ for all $n$ sufficiently large.

\begin{ex} \label{Example:SignedPermMatrices} {\bf (Signed permutation matrices: A first example of a character polynomial).} As an elementary example of a sequence of $B_n$--representations described by a character polynomial, consider the canonical action of the hyperoctahedral groups $B_n$ on the vector space $\Q^n$ by signed permutation matrices. The trace of a signed permutation matrix $\s$ is 
\begin{align*}
 \mathrm{Tr}(\s) &=  \text{$\# $ \{$1$'s on the diagonal of $\s$\}} - \text{$\# $ \{$(-1)$'s on the diagonal of $\s$\}} \\ 
  &= \text{ $\# $ \{ positive one cycles of $\s$\} } - \text{ $\# $ \{ negative one cycles of $\s$\} } \\
  &= X_1(\s) - Y_1(\s)
\end{align*}
and so the characters $\chi_n$ of this sequence are given by the function $\chi_n = X_1 - Y_1$ for all values of $n$.

The group $D_n$ is canonically realized as the subgroup of this signed permutation matrix group comprising those matrices with an even number of entries equal to $(-1)$. The character of this representation is the restriction of the character $\chi_n$ to the subgroup $D_n \subseteq B_n$, and so again this sequence of characters is equal to the character polynomial $\chi_n = X_1 - Y_1$ for all values of $n$. \end{ex} 

Conjugacy classes of the groups $D_n \subseteq B_n$ are not fully classified by their signed cycle type, due to the existence of certain 'split' classes when $n$ is even; see Section \ref{RepTheoryDnReview} for details. The functions $\{X_r, Y_r\}$ therefore do not span the space of class functions on any group $D_n$ with $n$ even. We prove, however, that when a sequence of representations $\{ V_n \}$ of $D_n$ has the structure of a finitely generated $\FI_D$--module, for $n$ large the characters depend only on the signed cycle type of the classes. Remarkably, the characters associated to $\{V_n\}$ are, for $n$ large,  also equal to a character polynomial independent of $n$.

\newtheorem*{CharacterPolynomials}{Theorem \ref{WnPersinomial}}
\begin{CharacterPolynomials}
{\bf (Characters of finitely generated $\FI_{\W}$--modules are eventually polynomial).}
Let $k$ be a field of characteristic zero. Suppose that $V$ is a finitely generated $\FI_{BC}$--module with weight $\leq d$ and stability degree $\leq s$, or,  alternatively, suppose that $V$ is a finitely generated $\FI_D$--module with weight $\leq d$ such that $\Ind_{D}^{BC}\;V$ has stability degree $\leq s$. In either case, there is a unique polynomial $$F_V \in k[X_1, Y_1, X_2, Y_2, \ldots],$$ independent of $n$, such that the character of $\W_n$ on $V_n$ is given by $F_V$ for all $n \geq s +d$. The polynomial $F_V$ has degree $\leq d$, with deg$(X_i)=$deg$(Y_i)=i$.
\end{CharacterPolynomials}

\noindent Weight and stability degree are defined in Sections \ref{SectionWeightReview} and \ref{SectionCoinvariantsStabilityDegreeReview}; these quantities are always finite for finitely generated $\FIW$--modules and associated induced $\FIW$--modules. 

Theorem \ref{WnPersinomial} generalizes the result of Church--Ellenberg--Farb \cite[Theorem 2.67]{CEF} that the characters of finitely generated $\FI_A$--module are, for $n$ sufficiently large, given by a character polynomial in the class functions $X_r$ on $\coprod_{n=0}^{\infty} S_n$ that takes a permutation $\sigma$ and returns the number of $r$--cycles in its cycle type.

In our applications, it remains an open problem to compute the character polynomials in all but a few small degrees. Since we can often establish explicit upper bounds on the degrees and stable ranges of these polynomials, the problem is much more tractable: to find the character polynomials  -- and so determine the characters for all values of $n$ -- it is enough to compute the characters for finitely many specific values of $n$. \\

{ \noindent \bf Eventually polynomial dimensions. \quad} Suppose that $V$ is a finitely generated $\FIW$--module with character polynomial $F_V$. For each $n$ in the stable range, the dimension dim$(V_n)$ is given by $F_V(n,0,0,0, \ldots),$ the value of the character polynomial on the identity element in $\W_n$. This has the immediate consequence:

\newtheorem*{polygrowth}{Corollary \ref{PolyGrowth}}
\begin{polygrowth} {\bf (Polynomial growth of dimension).}
 Let $V$ be an $\FIW$--module over a field of characteristic zero, and suppose $V$ is finitely generated in degree $\leq d$. Then for large $n$, dim$(V_n)$ is equal to a polynomial in $n$ of degree at most $d$. Equality holds for $n$ in the stable range given in Theorem \ref{WnPersinomial}.
\end{polygrowth}

Although our results on character polynomials in general hold only over fields of characteristic zero, this ``eventually polynomial'' growth of dimension holds even over positive characteristic, as we show in Theorem \ref{PolyDimCharP} (stated above). Our proof of Theorem \ref{PolyDimCharP} uses result in type A proven by Church--Ellenberg--Farb--Nagpal \cite[Theorem 1.2]{CEFN}.

\subsection{$\FIW\sh$--modules} 
In Section \ref{SectionFISharp} we study \emph{$\FI_{BC}\sh$--modules}, a class of $\FI_{BC}$--modules with additional structure: the $\FI_{BC}$ morphisms admit partial inverses. See Definition \ref{DefnFISharp} for a complete description. $\FI_{BC}\sh$--modules mirror the $\FI\sh$--modules (``FI sharp modules'') introduced by Church--Ellenberg--Farb \cite{CEF} in type A. 

In Example \ref{FIModuleExamples}, the $\FI_{BC}$--modules in number \ref{FIExample-TrivialReps}, \ref{FIExample-SignedPermMatrices}, \ref{FIExample-Schur}, and \ref{FIExample-Induced} all have $\FI_{BC}\sh$--module structures. Number \ref{FIExample-ZeroMaps} and \ref{FIExample-Truncated} cannot be promoted to $\FI_{BC}\sh$--modules. 

The structure of a finitely generated $\FI_{BC}\sh$--module is highly constrained. We prove in Theorem \ref{ClassifyingFISharp} that $\FI_{BC}\sh$--modules can be decomposed as direct sums of $\FIW$--modules of the form
$$ \left\{ \bigoplus_{m=0} \Ind_{B_{m}\times B_{n-m}}^{B_n} U_m \boxtimes k \right\}_n = \bigoplus_{m=0} M_{\W}(U_m). $$ \noindent As in Example \ref{FIModuleExamples}.\ref{FIExample-Induced},  $k$ denotes the trivial $B_{n-m}$--representation, and $U_m$ is a $B_m$--representation, possibly $0$. The external tensor product $(U_m \boxtimes k)$ is the $k$--module $(U_m \otimes_k k)$ as a $(B_{m}\times B_{n-m})$--representation. This classification result parallels a corresponding statement for $\FI\sh$--modules proven by Church--Ellenberg--Farb \cite[Theorem 2.24]{CEF}. 

Some consequences of this additional structure: an $\FI_{BC}\sh$--module finitely generated in degree $\leq d$ has characters equal to a unique character polynomial of degree at most $d$ for \emph{all} values of $n$, and dimensions given by a polynomial in $n$ of degree at most $d$ for all $n$. 
Additional consequences are given in Section \ref{CharacterPolynomialsFISharp}. 


\subsection{Some applications}

$\FIW$--modules and $\FIW\sh$--modules arise naturally throughout geometry and topology, and in Section \ref{SectionApplications} we use the theory developed here to give results for two such families: the cohomology groups of the pure string motion group $\PS_n$, and the cohomology groups of the hyperplane complements associated to the reflection groups $\W_n$. \\

 \noindent {\bf  Application: the pure string motion group.  \quad} Let $\PS_n$ be the group of  \emph{pure string motions}, motions of $n$ disjoint, unlinked, unknotted, smoothly embedded circles $S^1$ in $\R^3$. This motion group is a generalization of the pure braid group, and can be realized as the group of \emph{pure symmetric automorphisms} of the free group $F_n$; see Section \ref{SectionPureStringMotion} for a complete definition. 
 
 The pure string motion group also appears in the literature under the names the \emph{group of loops}, the \emph{pure untwisted ring group},  the group of \emph{basis-conjugating automorphisms} of the free group, the \emph{Fouxe-Rabinovitch automorphism group} of the free group, and the \emph{Whitehead automorphism group} of the free group. For more background on these groups and their cohomology, see for example Brendle--Hatcher \cite{BrendleHatcher}, Brownstein--Lee \cite{BrownsteinLee}, Dahm \cite{Dahm}, Goldsmith \cite{Goldsmith}, Jensen--McCammond--Meier \cite{JMM}, McCool \cite{McCool}, or Wilson \cite{PureStringMotion}.

\newtheorem*{PureStringMotionSharp}{Theorem \ref{PureStingMotionisFI-Sharp}} \begin{PureStringMotionSharp}
Let $k$ be $\Z$ or $\Q$. The cohomology rings $ H^*(\PS_{\bullet},  k)$ form an $\FI_{BC} \sh$--module, and a graded $\FI_{BC}$--algebra of finite type, with  $ H^m(\PS_{\bullet},  k)$ finitely generated in degree $\leq 2m$. In particular the $\FI_{BC}$--algebra $ H^*(\PS_{\bullet},  \Q)$ has slope $\leq 2$.
\end{PureStringMotionSharp}

We recover (with considerably less effort) the main result of our previous paper \cite{PureStringMotion}:

\newtheorem*{StringMotionRepStability}{Corollary \ref{PSnRepStability} }
\begin{StringMotionRepStability} For each $m$, the sequence $ \{ H^m(\PS_{n};\Q)\}_n$ of representations of $B_n$ (or $S_n$) is uniformly representation stable, stabilizing once $n \geq 4m$.
\end{StringMotionRepStability}

A consequence of uniform representation stability, which follows from stability for the trivial representation and a transfer argument, is rational homological stability for the string motion group $\S_n$. This recovers the rational case of a result of Hatcher and Wahl \cite[Corollary 1.2]{HatcherWahl}. More details are given in Section 7 of \cite{PureStringMotion}.

Another consequence of Theorem \ref{PureStingMotionisFI-Sharp} is the existence of character polynomials. Because these cohomology groups are $\FI_{BC} \sh$--modules, their characters are equal to the character polynomial for all values of $n$, and not just $n$ sufficiently large. 

\newtheorem*{StringMotionCharPoly}{Corollary \ref{PSnCharPoly}}
\begin{StringMotionCharPoly} Let $k$ be $\Z$ or $\Q$. Fix an integer $m \geq 0$. The characters of the sequence of $B_n$--representations $\{ H^m(\PS_{n}; k) \}_n$ are given, for all values of $n$, by a unique character polynomial of degree $\leq 2m$. 
\end{StringMotionCharPoly}

We compute these character polynomials explicitly in degree $1$ and $2$:
\begin{align*} 
\chi_{H^1(\PS_{\bullet}; \Z)}  & = X_1^2 -X_1 - Y_1^2 + Y_1 \\
\chi_{H^2(\PS_\bullet; \Z)}  & =  \; 2 X_2 + Y_1^2 + 2 Y_2^2 - X_1^2 Y_1^2 - \frac32 Y_1^3 + \frac12 Y_1^4 + X_1^2 - 2X_2^2 -\frac32 X_1^3  \\ & + \frac12 X_1^4 + \frac12 X_1 Y_1^2 - X_1  Y_2 -X_2 Y_1 - Y_1 Y_2 + \frac12 X_1^2 Y_1 - X_1 X_2 - 2Y_2  
\end{align*}

\noindent It is an open problem to compute these polynomials for larger values of $m$. \\

\noindent  {\bf  Application: hyperplane complements. \quad} Each family of groups $\W_n$ has a canonical action on $\R^n$ by signed permutation matrices; we denote by $\cA_{\W}(n)$ the set of complexified hyperplanes fixed by reflections in $\W_n$, and $$\cM_{\W}(n) := \C^n \bigg\backslash \bigcup_{H \in \cA(n)} H$$ the associated hyperplane complement. See Section \ref{SectionHyperplaneComplements} for explicit descriptions of these spaces, and a brief survey of results on the structure of their cohomology rings. In type A, the space $\cM_{A}(n)$ is precisely the ordered $n$-point configuration space of $\C$, and Church--Ellenberg--Farb show its cohomology groups are finitely generated $\FI_A\sh$--modules \cite[Theorem 4.7]{CEF}.  Using a presentation for $H^*(\cM_{\W}(n); \C)$ computed by Brieskorn \cite{Brieskorn} and Orlik--Solomon \cite{OrlikSolomon}, we generalize the results of \cite{CEF} to all three families of classical Weyl groups. 

\newtheorem*{HyperplaneComplementSharp}{Theorem \ref{HyperplaneCohomologyIsFISharp}} \begin{HyperplaneComplementSharp}
Let $\cM_{\W}$ be the complex hyperplane complement associated with the Weyl group $\W_n$ in type A$_{n-1}$, B$_n$/C$_n$, or D$_n$. In each degree $m$, the groups $H^m(\cM_{A}(\bullet), \C)$ form an $\FI_A\sh$--module finitely generated in degree $\leq 2m$, and both $H^m(\cM_{BC}(\bullet), \C)$ and $H^m(\cM_{D}(\bullet), \C)$ are $\FI_{BC}\sh$--modules finitely generated in degree $\leq 2m$. 
  \end{HyperplaneComplementSharp}
  
  \newtheorem*{hyperplaneRepStability}{Corollary \ref{HypRepStable}}
  \begin{hyperplaneRepStability}
  In each degree $m$, the sequence of cohomology groups $\{H^m(\cM_{\W}(n), \C)\}_n$ is uniformly representation stable in degree $\leq 4m$. 
\end{hyperplaneRepStability}

\newtheorem*{hyperplaneCharPoly}{Corollary \ref{HyperplaneComplementCharPoly}}
\begin{hyperplaneCharPoly} 
  In each degree $m$, the sequence of characters of the $\W_n$--representations $H^m(\cM_{\W}(n), \C)$ are given by a unique character polynomial of degree $\leq 2m$ for all $n$. 
\end{hyperplaneCharPoly}

\noindent We emphasize that, because these sequences are $\FIW\sh$--modules, their characters are equal to the character polynomial for \emph{every} value of $n$. 

Corollary \ref{HypRepStable} recovers the work of Church--Farb \cite[Theorem 4.1 and 4.6]{RepStability} in types A and B/C. In type A, Theorem \ref{HyperplaneCohomologyIsFISharp} follows from the work of Church--Ellenberg--Farb \cite{CEF} on the cohomology of the ordered configuration space of the plane.

Character polynomials and stable decompositions for $H^m(\cM_{A}(\bullet), \C)$ are computed in \cite{CEF} for some small values of $m$. In Type B/C and D, we can also compute the character polynomials by hand in small degree:
\begin{align*} 
\chi_{H^1(\cM_{D}(\bullet), \C)}   = 2 { X_1 \choose 2 } + 2 { Y_1 \choose 2 } + 2 X_2 \qquad 
\chi_{H^1(\cM_{BC}(\bullet), \C)}   = 2 { X_1 \choose 2 } + 2 { Y_1 \choose 2 } + 2 X_2 + X_1 - Y_1 
\end{align*}
\noindent See Section \ref{SectionHyperplaneComplements} for the character polynomials and stable decompositions in degree $m$ is $1$ and $2$.

 \subsection{Relationship to the theory of FI--modules} \label{SectionEarlierWork}

Our theory of $\FIW$--modules has very close parallels to work of Church, Ellenberg, Farb, and Nagpal \cite{CEF, CEFN} on the symmetric groups, which we aim to highlight throughout this paper. Working with the other classical Weyl groups, however, we do encounter some new obstacles and some new phenomena. We enumerate some differences here:

\noindent {\bf Character polynomials in type B/C. \quad } The existence of character polynomials for finitely generated $\FI_A$--modules follows immediately from representation stability and classical results in algebraic combinatorics: the formula for the character polynomial of the irreducible $S_n$--representation $V(\y)_n$ appear in texts such as MacDonald \cite{MacdonaldSymmetric}. The achievement of \cite{CEF} here was uncovering this (regrettably little-known) formula and recognizing its implications for the study of $\FI_A$--modules. The analogous formulas for the irreducible $B_n$--representations are less readily available, however, and we compute these in Section \ref{SectionCharPolysBCD}. These signed character polynomials now involve two sets of variables $X_r$ and $Y_r$, corresponding to the positive and negative cycles for these groups. \\

\noindent {\bf Character polynomials in type D. \quad } Given the classification of conjugacy classes in type D (Section \ref{RepTheoryDnReview}), and the existence of 'split' classes that could not be characterized by signed cycle type, we had not expected an analogue of character polynomials to exist for sequences of $D_n$--representations, except in exceptional cases. A finitely generated $\FI_D$--module \emph{does} have characters equal, for large $n$, to a character polynomial. We establish this existence result by realizing the tail of a finitely generated $\FI_D$--module $V$ as the restriction of an $\FI_{BC}$--module, using properties of categorical induction $\Ind_{D}^{BC}$. \\
 
\noindent  {\bf A category $\FI_D\sh$? \quad } There does not appear to be a suitable analogue of $\FI\sh$ for the category $\FI_D$; see Remark \ref{RemarkFIDSharp}.  Fortunately, and perhaps not by coincidence, the applications in type D where we have expected this extra structure, such as the cohomology groups of the hyperplane complements $\cM_D(n)$, turned out to be restrictions of $\FI_{BC}\sh$--modules to $\FI_D \subseteq \FI_{BC}$.

\subsection*{Acknowledgments}

I am grateful to Benson Farb, Tom Church, Jordan Ellenberg, Rita Jimenez Rolland, and Peter May for numerous helpful discussions about this project. Thanks are especially due to Benson Farb, my PhD advisor, for extensive guidance and feedback throughout the project. 

This work was supported in part by a PGS D Scholarship from the Natural Sciences and Engineering Research Council of Canada.

\section{Background} \label{SectionBackgroundReview}
\subsection{The Weyl groups of classical type and their representation theory} \label{SectionRepTheoryReview}

We briefly summarize the rational representation theory of the three families of Weyl groups, and the associated notation used in this paper. A more detailed review, with additional references, is given in \cite[Section \ref{FIW1-SectionBackground}]{FIW1}.

\subsubsection{The symmetric group $S_n$} \label{RepTheorySnReview}

Given a partition $\y \vdash n$, $\y = ( \y_0, \y_1, \ldots, \y_r)$, we denote the \emph{parts} of the partition by $\y_i$ and index them in decreasing order $\y_0 \geq \y_1 \geq \cdots \geq \y_r$. We write $|\y|=n$ to indicate the size of the partition, and write $\ell(\y)$ to denote the \emph{length} of $\y$, the number of parts. 

The rational irreducible representations of the symmetric group $S_n$ are classified by partitions of $n$ (see for example Fulton--Harris \cite{FultonHarris}) and we write $V_{\y}$ to denote the $S_n$--representation associated to the partition $\y$.

 Given a partition $\y \vdash m$ and an integer $n \geq \y_1+m$, define the \emph{padded partition} $$\y[n]:= ( (n-m) , \y_1, \y_2, \ldots, \y_t).$$  We denote by $V(\y)_n$ the irreducible $S_n$--representation associated to $\y[n]$, that is,
\begin{align*}
 V(\y)_n &:= \left\{ \begin{array}{ll}
         V_{\y[n]} & \mbox{$(n-m)\geq \y_1$},\\
        0 & \mbox{otherwise}.\end{array} \right.
\end{align*}

\subsubsection{The hyperoctahedral group $B_n$} \label{RepTheoryBnReview}

Recall that the hyperoctahedral group (or signed permutation group) $B_n$ is the group of generalized permutation matrices with nonzero entries $\pm 1$; equivalently, $B_n=(\Z / 2\Z)^n \rtimes S_n$ is the symmetry group of the set $\{ \{-1,1\}, \{-2, 2\}, \ldots, \{-n, n\} \}.$ We frequently consider $B_n$ as  a subgroup of the symmetric group on the set $$\Omega = \{-1, 1, -2, 2, \ldots, -n, n \},$$ and write signed permutations in the corresponding cycle notation.\\

{ \noindent \bf  The rational representation theory of $B_n$.  \quad}  Recall that the irreducible rational $B_n$--representations are classified by \emph{double partitions} of $n$, ordered pairs of partitions $\y = (\y^+, \y^-) \text{ with}$ $|\y^+|+|\y^-|=n,$ as follows. 

 Let $V_{( \y^+, \varnothing)}$ denote the $B_n$--representation pulled back from $S_n$--representation $V_{\y^+}$, and denote $$V_{( \varnothing, \y^-)}  : = V_{( \y^-, \varnothing)} \otimes \Q^{\varepsilon},$$ where $\Q^{\varepsilon}$ is the one-dimensional representation given by the character $\varepsilon: B_n \to B_n / D_n = \{\pm 1 \}$. Then for $\y^+ \vdash m$ and $\y^- \vdash (n-m)$ we define 
$$ V_{(\y^+, \y^-)} := \Ind_{B_m \times B_{n-m}}^{B_n} V_{(\y^+, \varnothing)} \boxtimes V_{(\varnothing,\y^-)}, $$ 
where $\boxtimes$ again denotes the external tensor product of the $B_m$--representation $V_{(\y^+, \varnothing)}$ with the $B_{n-m}$--representation $V_{(\varnothing,\y^-)}$. The rational irreducible $B_n$--representations are precisely the set $$\left\{ V_{(\y^+, \y^-)} \; \mid \; (\y^+, \y^-) \text{ is a double partition of $n$} \right\}.$$

For a double partition $\y = ( \y^+, \y^-)$ with $\y^+ \vdash \ell $ and $\y^- \vdash m $, we define the \emph{padded double partition} $\y[n] : = ( \y^+[n-m], \y^-)$ associated to $\y=(\y^+, \y^-)$. We write $V( \y )_n$ or $V( \y^+, \y^-)_n$ to denote the irreducible $B_n$--representation 
\begin{align*}
 V(\y)_n &:= \left\{ \begin{array}{ll}
         V_{\y[n]} & \mbox{$(n-m)\geq \y^+_1$},\\
        0 & \mbox{otherwise}.\end{array} \right.
\end{align*}

{ \noindent \bf  The conjugacy classes of $B_n$.  \quad} The conjugacy classes of $B_n$ are classified by \emph{signed cycle type}. 
Each element of $B_n$ decomposes into a product of \emph{cycles}; a factor is called an \emph{$r$--cycle} if it maps to an $r$--cycle in $S_n$ under the natural surjection. An $r$--cycle is \emph{positive} if its $r^{th}$ power is the identity, equivalently, if it reverses the sign of an even number of digits 
$\{\pm 1, \ldots, \pm n\}$.  An $r$--cycle is \emph{negative} if its $r^{th}$ power is the product of $r$ involutions $(-i \; i )$, equivalently, if it reverses the sign of an odd number of digits. For example, $(1)(-1)$ is a positive $1$--cycle, and $(-1 \; 1)$ is a negative $1$--cycle. 

We designate the cycle type of a signed permutation in $B_n$ by a double partition $(\nu^+, \nu^-)$ of $n$, where the parts of $\nu^+$ are the lengths of the positive cycles and the parts of $\nu^-$ are the lengths of the negative cycles.  

\subsubsection{The even-signed permutation group $D_n$} \label{RepTheoryDnReview}

The even-signed permutation group $D_n$ is the subgroup of signed permutation matrices $B_n$ of matrices that have an even number of entries equal to $-1$. 

{ \noindent \bf  The rational representation theory of $D_n$.  \quad} 

Given an irreducible representation $V_{(\y^+, \y^-)}$ of $B_n$ with $\y^+ \neq \y^-$, its restriction to $D_n$ is also irreducible. We denote this irreducible $D_n$--representation by

$$ V_{ \{\y^+, \y^- \}} := \Res^{B_n}_{D_n} V_{(\y^+, \y^-)}  \cong   \Res^{B_n}_{D_n} V_{( \y^-, \y^+)}, \qquad \text{$\y^+ \neq \y^-$}.$$

These $D_n$--representations are nonisomorphic for each distinct set of partitions $\{\y^+, \y^- \}$. When $n$ is even, for any partition $\y \vdash \frac{n}{2}$, the irreducible $B_n$--representation $V_{(\y, \y)}$ restricts to a sum of two nonisomorphic irreducible $D_n$--representations of equal dimension.  The irreducible rational representations of $D_n$ are therefore classified by the set $$ \left\{ \{ \y^+, \y^- \} \; \mid \; \y^+ \neq \y^-, \; |\y^+|+|\y^-| = n \right\} \coprod   \left\{ (\y, \pm) \; \middle| \; |\y|=\frac{n}{2} \right\}, $$ The 'split' irreducible representations $V_{(\y, +)}$ and $V_{(\y, -)}$ only arise when $n$ is even. 

 Given an (ordered) double partition $\y = ( \y^+, \y^-)$ with $\y^+ \vdash \ell $ and $\y^- \vdash m $, we write $V( \y )_n$ to denote the $D_n$--representation $V(\y)_n := \Res_{D_n}^{B_n} V(\y)_n.$ Explicitly, $V(\y)_n$ is the $D_n$--representation
\begin{align*}
 V(\y)_n &= \left\{ \begin{array}{ll}
         V_{ \{ \y^+[n-m], \; \y^- \}} & \mbox{$(n-m)\geq \y^+_1$ and $\y^+[n-m] \neq \y^-$},\\
         V_{ \{  \y^-, \; + \}} \oplus V_{ \{  \y^-, \; - \}} & \mbox{$(n-m)\geq \y^+_1$ and $\y^+[n-m] = \y^-$}, \\
        0 & \mbox{otherwise}.\end{array} \right.
\end{align*}
 The $D_n$--representation $V(\y)_n$ is irreducible for all but at most one value of $n$. 
 
 We remark that, in contrast to the sequence of $S_n$ or $B_n$ representations $V(\y)_n$, knowing the $D_n$--representation $V(\y)_n$ for a single value of $n$ may not be enough to determine $V(\y)_{n+1}$, as we cannot distinguish the partitions $\y^+[n-m]$ and $\y^-$. 

{ \noindent \bf  The conjugacy classes of $D_n$. \quad}  As with $B_n$, each element of $D_n$ decomposes into a product of signed cycles; by definition each element must have an even number of negative cycles. Signed cycle type is a $D_n$ conjugacy class invariant, and it nearly classifies the conjugacy classes, with one qualification: when $n$ is even, the elements for which all cycles are positive and have even length are now split between two conjugacy classes.

\subsection{Representation stability} \label{SectionRepStabilityReview}

Representation stability was introduced by Church--Farb \cite{RepStability} for a variety of families of groups $G_n$ including $S_n$ and $B_n$. In \cite[Section \ref{FIW1-BackgroundRepStability}]{FIW1} we additionally defined representation stability for sequences of $D_n$--representations. We recall these definitions here.

\begin{defn} \label{DefnRepStability} {\bf (Representation stability).}
Let $\W_n$ be one of the families of classical Weyl groups, and suppose $\{ V_n \}$ is a sequence of finite-dimensional $\W_n$--representations over characteristic zero with maps $\phi_n :V_n \to V_{n+1}$. The sequence $\{V_n, \phi_n \}$ is \emph{consistent} if $\phi_n$ is equivariant with respect to the action of $\W_n$ on $V_n$ and of $\W_n \subseteq \W_{n+1}$ on $V_{n+1}$. 

A consistent sequence $\{V_n, \phi_n\}$ is \emph{representation stable} if it satisfies three properties: \\[-10pt]

\noindent I. \textbf{\emph{Injectivity.}} The maps $\phi_n :V_n \to V_{n+1}$ are injective for $n>>0$. \\[-10pt]

\noindent II. \textbf{\emph{Surjectivity.}} The image $\phi_n(V_n)$ generates $V_{n+1}$ as a $k[\W_{n+1}]$--module for $n>>0$. \\[-10pt]

\noindent III. \textbf{\emph{Multiplicities.}} Decompose $V_n$ into irreducible $\W_n$--representations: $$V_n = \bigoplus_{\y}c_{\y,n}V(\y)_n.$$ For each $\y$ there exists some $N_{\y}$ such that the multiplicity $c_{\y,n}$ of $V(\y)_n$ is constant for $n \geq N_{\y}$. \\[-10pt]

The sequence is \emph{uniformly} representation stable if $N_{\y}$ can be chosen independently of $\y$.

\end{defn}

A main result of \cite{FIW1} is that for rational $\FIW$--modules, finite generation is equivalent to uniform representation stability.

\begin{thm} \emph{\cite[Theorem \ref{FIW1-FinGenIffRepStable}]{FIW1}} 
 Let $k$ be a field of characteristic zero. An $\FIW$--module $V$ is finitely generated if and only if $\{V_n \}$ is uniformly representation stable with respect to the maps induced by the natural inclusions $I_n : \bn \hookrightarrow {\bf (n+1)}$. 
\end{thm}

Details (including bounds on the stable range) are given in Section \ref{SectionRepStabilityforFIModules}.

\subsection{Summary of terminology and foundations for $\FIW$--modules} \label{SectionFIWFoundationsReview}

In this section we summarize the main definitions and foundational results developed in \cite{FIW1} on $\FIW$--module theory.

\subsubsection{The $\FIW$--modules $M_{\W}(\bm)$ and $M_{\W}(U)$} \label{SectionMW}

\begin{defn} {\bf (The $\FIW$--modules $M_{\W}(\bm)$ and $M_{\W}(U)$).} \label{Defn:M(m)Review} Let $\W_n$ denote $S_n$, $B_n$, or $D_n$. For fixed integer $m \geq 0$, define $M_{\W}(\bm)$ to be the $\FIW$--module over $k$ such that $$M_{\W}(\bm)_n := k \big[ \Hom_{\FIW}( \bm, \bn)\big]$$ with an action of $\W_n$ by postcomposition. 

We can identify $M_{\W}(\bm)_n$ with the $k$--span of the set $$ \left\{ \Big(f(1), f(2), \ldots, f(m)\Big) \subseteq \bn \; \mid \; f: \bm \to \bn \text{ an  $\FIW$ morphism} \right\}. $$ For each $n$ we have an isomorphism of $\W_n$--represntations $$M_{\W}(\bm)_n \cong \Ind_{\W_{n-m}}^{\W_n} k;$$ here $k$ is the trivial $\W_m$--representation.

Given a $\W_m$--representation $U$, we define $M_{\W}(U)$ to be the $\FIW$--module $$M_{\W}(U)_n := M_{\W}(\bm)_n \otimes_{k[\W_m]} U.$$ In particular $M_{\W}(\bm) \cong M_{\W}(k[\W_m])$. 

Let FB$_\W$ denote the wide subcategory of $\FIW$ consisting of all objects and all endomorphisms. Denote by FB$_\W$--Mod the category of functors $$ \text{FB}_{\W} \longrightarrow k\text{--Mod};$$ the objects are sequences of $\W_m$--representations (with no additional maps) and the morphisms are sequences of $\W_m$--equivariant maps.

We extend $M_{\W}$ to a functor on the category FB$_\W$--Mod 

\begin{align*}
M_{\W}: \text{FB$_\W$--Mod} & \longrightarrow \FIW \text{--Mod} \\
  U_m & \longmapsto   \mu_m (U_m).
\end{align*}

\end{defn}

Given an $\FIW$--module $V$ and any subset $S = \{v_i\} \subseteq \coprod_{n \geq 0} V_n$, with $v_i \in V_{m_i}$, there is a unique map of $\FIW$--module 
\begin{align*}
\bigoplus_{i=1} M_{\W}(\bm_i) & \longrightarrow V \\
f  & \longmapsto f_*(v_i)  \qquad f \in \Hom_{\W}(\bm_i, \bn), \text{ the basis for } M_{\W}(\bm_i)_n
\end{align*}
from $\bigoplus_{i=1} M_{\W}(\bm_i)$ onto the $\FIW$--submodule \emph{generated by S}, the smallest submodule of $V$ that contains $S$. An $\FIW$--module $V$ is finitely generated if and only if it is the quotient of a finite sum of $\FIW$--modules $\bigoplus_{i=1}^{\ell} M_{\W}(\bm_i)$.

\begin{defn}{\bf (Finite Presentation; Relation degree).}
 Let $V$ be a finitely generated $\FIW$--module. Then $V$ is \emph{finitely presented} with \emph{relation degree} $r$ if there is a surjection $$ \bigoplus_{i=1}^{\ell} M_{\W}(\bm_i) \twoheadrightarrow V $$ with a kernel finitely generated in degree at most $r$. 
 
 The Noetherian property \cite[Theorem \ref{FIW1-Noetherian}]{FIW1} implies that all finitely generated $\FIW$--modules are finitely presented. 
\end{defn}

\subsubsection{The Noetherian property} \label{SectionNoetherianReview}

A critical property of $\FIW$--modules that underlies all our major results is that the category of $\FIW$--modules over a Noetherian ring is Noetherian.  Church--Ellenberg--Farb--Nagpal prove this result for $\FI_A$--modules \cite[Theorem 1.1]{CEFN}, and we use their work to prove it more generally:

\begin{thm} \label{NoetherianReview} \emph{\cite[Theorem \ref{FIW1-Noetherian}]{FIW1}} {\bf ($\FIW$--modules are Noetherian).}  Let $k$ be a Noetherian ring. Any sub--$\FIW$--module of a finitely generated $\FIW$--module over $k$ is itself finitely generated.
\end{thm}

\subsubsection{The functor $H_0$}

\begin{defn} \label{DefnH0Review} {\bf (The functor $H_0$).} As in \cite[Definition 2.18]{CEF}, we define the functor $H_0$ by
\begin{align*}
H_0 & : \FIW \text{--Mod}   \longrightarrow \text{FB$_\W$--Mod}\\
& (H_0(V))_n =  V_n \bigg/ \bigg(\text{span}_V \big(  \coprod_{k<n} V_k \big) \bigg)_n
\end{align*}
$H_0(V)$ is a minimal set of $\W_m$--representations to generate the $\FIW$--module $V$. We can put an $\FIW$--module structure on the $\W_n$--representations $(H_0(V))_n$ by letting $I_n$ act by $0$ for all $n$; it is the largest quotient of $V$ where all $\FIW$ morphisms $f$ between distinct objects act by zero. We denote this $\FIW$--module by $H_0(V)^{\FIW}$.

$H_0$ is a left inverse to $M_{\W}$, that is, given $U = \{ U_m \}_m$ we have $H_0(M_\W(U))_m = U_m$ for all $m$. We will see in Section \ref{SectionFISharp} that additionally $M_{\W}(H_0(V)) = V$ when $V$ has the additional structure of an \emph{$\FIW \sh$--module}. 

There are surjections $$M_{\W} (H_0(V)) \twoheadrightarrow V \qquad \text{ and } \qquad V \twoheadrightarrow H_0(V)^{\FIW}.$$

\end{defn}

\subsubsection{Restriction and induction of $\FIW$--modules} \label{SectionReviewRestrictionInduction}

The inclusions of categories $\FI_A \subseteq \FI_D \subseteq \FI_{BC}$ enable us to define restriction and induction operations on the corresponding categories of $\FIW$--modules. Both restriction and induction preserve finite generation of $\FIW$--modules, a fact that we use to recover results in type B/C and D from work of Church--Ellenberg--Farb--Nagpal \cite{CEFN} in type A. We use additional properties of these operations to develop theory for $\FI_D$--modules using results in type B/C. 

\begin{defn}{\bf (Restriction).}
 Consider Weyl group families $\W_n$ and $\oW_n$ with $\W_n \subseteq \oW_n$. Then there is an inclusion $\FIW \hookrightarrow \oFIW$, and  given any $\oFIW$--module $V$ we denote by $\Res_{\W}^{\oW} V$ the $\FIW$--module obtained by restricting the functor $V$ to the subcategory $\FIW$.
\end{defn}

\begin{prop} \label{RestrictionPreservesFinGenReview} \emph{\cite[Proposition \ref{FIW1-RestrictionPreservesFinGen}]{FIW1}} {\bf (Restriction preserves finite generation).}
For each family of Weyl groups $\W_n \subseteq \oW_n$, the restriction $\Res_{\W}^{\oW} V$ of a finitely generated $\oFIW$--module $V$ is finitely generated as an $\FIW$--module. Specifically, 
\begin{enumerate}
 \item \label{WntoSnReview} Given an $\FI_{BC}$--module $V$ finitely generated in degree $\leq m$, $\Res^{BC}_{A} V$ is finitely generated as an $\FI_A$--module in degree $\leq m$.
 \item \label{WntoW'nReview} Given an $\FI_{BC}$--module $V$ finitely generated in degree $\leq m$, $\Res^{BC}_{D} V$ is finitely generated as an $\FI_D$--module in degree $\leq m$.
\item \label{W'ntoSnReview} Given an $\FI_D$--module $V$ finitely generated in degree $\leq m$, $\Res^{D}_{A} V$ is finitely generated as an $\FI_A$--module in degree $\leq (m+1)$.
\end{enumerate}
\end{prop}

For categories $\FIW \hookrightarrow \oFIW$, the restriction functor $$\Res_{\W}^{\oW}: \oFIW \text{--Mod} \longrightarrow \FIW \text{--Mod}$$ has a left adjoint, the induction functor $$\Ind_{\W}^{\oW}: \FIW \text{--Mod} \longrightarrow \oFIW \text{--Mod}.$$

We recall that the adjunction comes with a \emph{unit} map, a natural transformation
\begin{align*}
  \eta:  id & \longrightarrow (\Res_{\W}^{\oW}  \Ind_{\W}^{\oW} ) \\ 
  \eta_{V}:  V & \longrightarrow \Res_{\W}^{\oW} ( \Ind_{\W}^{\oW} \; V )
\end{align*}
and this map defines a natural bijection
\begin{align*} \left\{ \begin{array}{c} \text{ $\FIW$\;--Module Maps } \\ V  \longrightarrow \Res_{\W}^{\oW} \; U \end{array} \right\} \quad \longleftrightarrow \quad \left\{ \begin{array}{c} \text{ $\oFIW$--Module Maps} \\ \Ind_{\W}^{\oW} \;V  \longrightarrow U  \end{array} \right\} 
\end{align*}

We can describe the induced functors $ \Ind_{\W}^{\oW} \; V$ explicitly using the theory of Kan extensions; see Mac Lane \cite[Chapter 10.4]{MacLaneWorking}. 

\begin{defn}{\bf (Induction).} \label{DefnIndReview} Given an $\FIW$--module $V$, and an inclusion of categories $\FIW \hookrightarrow \FI_{\oW}$, we define the \emph{induced $\FI_{\oW}$--module} $\Ind_{\W}^{\oW} V$ by 
$$ (\Ind_{\W}^{\oW} V)_n = \bigoplus_{r \leq n} M_{\oW}({\bf r})_n \otimes_k V_r \bigg/  \langle \quad f\otimes g_*(v) = (f\circ g) \otimes v \quad \vert \quad \text{ $g$ is an $\FIW$ morphism} \rangle .$$
with the action of $h \in \Hom_{\oW}(\bm, \bn)$ by $h_*: g\otimes v \longmapsto (h\circ g) \otimes v.$
\end{defn}

Given categories $\FIW \subseteq \oFIW$ and fixed $m \geq 0$, there is an isomorphism of $\oFIW$\,--modules $$\Ind_{\W}^{\oW} M_{\W}(\bm) \cong M_{\oW}(\bm).$$

A critical property of induction from $\FI_D$ to $\FI_{BC}$ was proven in  \cite[Theorem \ref{FIW1-VisResIndV}]{FIW1}. 

\begin{thm}\label{VisResIndVReview}\emph{\cite[Theorem \ref{FIW1-VisResIndV}]{FIW1}} {\bf($V_n \cong (\Res_{D}^{BC}\Ind_{D}^{BC} V)_n$ for $n$ large).} Suppose $V$ is an $\FI_D$--module finitely generated in degree $\leq m$. Then $$ V_n \xrightarrow{\cong} (\Res_{D}^{BC}\; \Ind_{D}^{BC} V)_n$$ is an isomorphism of $D_n$--representations for all $n>m$. In particular, every finitely generated $\FI_D$--module $V$ is, for $n$ greater than its degree of generation, the restriction of an $\FI_{BC}$--module. 
\end{thm}

\subsubsection{The weight of an $\FIW$--module} \label{SectionWeightReview}

\begin{defn} \label{Defn:WeightReview} {\bf (Weight).} Let $k$ be a field of characteristic zero. An $\FI_A$--module $V$ over $k$ has \emph{weight} $\leq d$ if every irreducible constituent $V(\y)_n$ of $V_n$ has $|\y| \leq d$ for all $n \geq 0$. Similarly an $\FI_{BC}$--module $V$ over $k$ has \emph{weight} $\leq d$ if for all $n$ every irreducible constituent $V(\y)_n= V(\y^+, \y^-)_n$ of $V_n$ satisfies $|\y^+| + |\y^-| \leq d$. Finally we say an $\FI_D$--module $V$ has \emph{weight} $\leq d$ if the $\FI_{BC}$--module $\Ind_{D}^{BC} V$ does. 

An $\FIW$--module $V$ has \emph{finite weight} if it is of weight $\leq d$ for some integer $d \geq 0$. The minimum such $d$ is the \emph{weight} of $V$, denoted weight$(V)$. We say that the \emph{weight of a Young diagram $\y = (\y_0, \y_1, \ldots, \y_\ell)$} is $$\y_1 + \cdots + \y_{\ell} = |\y|-\y_0.$$ 
\end{defn}

For example, the $\FIW$--module $M_{\W}(\bm)$ has weight $m$. Since the weight of an $\FIW$--module $V$ is an upper bound on the weight of all of its quotients, this implies that the weight of an $\FIW$--module is bounded by its degree of generation. 

\begin{thm} \label{WnDiagramSizesReview} \cite[Theorem \ref{FIW1-WnDiagramSizes}]{FIW1} {\bf (Degree of generation bounds weight).}  Suppose that $V$ is an $\FIW$--module over a field of characteristic zero.
  If $V$ is finitely generated in degree $\leq m$, then weight($V$) $\leq m$. 
\end{thm}

\subsubsection{Shifted $\FIW$--modules, coinvariants, and stability degree} \label{SectionCoinvariantsStabilityDegreeReview}

We define the shift and coinvariant functors on $\FIW$--modules, and use these operations to define the stability degree of an $\FIW$--module. The stability degree featured in the proof of representation stability for finitely generated $\FIW$--modules, and in Section \ref{SectionCharPolys} it will be used to bound the stable range of the associated character polynomials.

\begin{defn} {\bf (Shifts $\amalg_{[-a]}$; Shifted $\FIW$--modules)}
For each $a \geq 0$, define the functor
\begin{align*}
\amalg_{[-a]} : \FIW & \longrightarrow \FIW \\
\bn & \longmapsto {\bf (n+a) } \\ 
\{ f: \bm \to \bn \} & \longmapsto \{ \amalg_{[-a]}(f) :{\bf (m+a)} \to {\bf (n+a) } \}
\end{align*}
where
\begin{align*}
 \amalg_{[-a]}(f) \text{ maps } \left\{ \begin{array}{l}
         d \mapsto d  \qquad \mbox{if $d \leq a$},\\
         (d+a) \mapsto (f(d)+a).  \end{array} \right.
\end{align*}  
Then for an $\FIW$--module $V: \FIW \to k$--Mod, we define the \emph{shifted} $\FIW$--module $S_{+a}V$ by $$S_{+a}V = V \circ \amalg_{[-a]}.$$ As a $\W_n$--representation, $(S_{+a}V)_n$ is the restriction of $V_{n+a}$ from $\W_{n+a}$ to the copy of $\W_n$ acting on $\{\pm(1+a), \ldots \pm(n+a)\} \subseteq {\bf (n+a) }$. 
\end{defn}

\begin{defn} {\bf (Coinvariants functor $\t$; The graded $k[T]$--module $\Phi_a(V)$).} We define the functor $\t$, as follows:
\begin{align*}
\t: \FIW\text{--Mod} & \longrightarrow \FIW\text{--Mod} \\
V & \longmapsto (\t V)_n = (V_n)_{\W_n} := k \otimes_{k[\W_n]} V_n 
\end{align*}
 The spaces $(\t V)_n$ are the $\W_n$--coinvariants of $V_n$, the largest quotient of $V_n$ on which $\W_n$ acts trivially. Over characteristic zero it is isomorphic to its invariant subspace $(V_n)^{\W_n}$.
 
 Then, for fixed integer $a \geq 0$, we define
\begin{align*}
  \Phi_a : \FIW \text{--Mod} & \longrightarrow k[T]\text{--Mod} \\ 
        V & \longmapsto \bigoplus_{n \geq 0} (\t \circ S_{+a} V)_n  = \bigoplus_{n \geq 0} (V_{n+a})_{\W_n}     
\end{align*}
$\Phi_a(V)$ is a $k[T]$--module under the action of $T$ by the maps  $(V_{n+a})_{\W_n} \to (V_{n+1+a})_{\W_{n+1}}$ induced by the maps $(I_{n+a})_* : V_{n+a} \to V_{n+1+a}$.  
\end{defn}

\begin{defn} {\bf (Injectivity degree; Surjectivity degree; Stability degree).} An $\FIW$--module $V$ has \emph{injectivity degree $ \leq s$} (respectively, \emph{surjectivity degree $ \leq s$}) if for all $a \geq 0$ and $n \geq s$, the map $\Phi_a(V)_n \to \Phi_a(V)_{n+1}$ induced by $T$ is injective (respectively, surjective). We call the minimum such $s$ the \emph{injectivity degree} (respectively, the \emph{surjectivity degree}).

We say that $V$ has \emph{stability degree $\leq s$} if for every $a \geq 0$ and $n \geq s$, the map $\Phi_a(V)_n \to \Phi_a(V)_{n+1}$ induced by $T$ is an isomorphism of $k$--modules, equivalently, an isomorphism of $\W_a$--representations. Explicitly, $V$ has stability degree $\leq s$ if $$(V_{n+a})_{\W_{n}} \cong (V_{n+1+a})_{\W_{n+1}} \qquad \text{for every $a \geq 0$ and all $n \geq s.$}$$

The \emph{stability degree} is the minimum such $s$.
\end{defn}

\subsubsection{Representation stability of finitely generated $\FIW$--modules} \label{SectionRepStabilityforFIModules}

Using the concept of stability degree, we prove in \cite[Section \ref{FIW1-SectionFinGenRepStability}]{FIW1} that for $\FIW$--modules over characteristic zero, finite generation is equivalent to uniform representation stability. Our first result:

 \begin{thm} \label{StabilityDegreeImpliesRepStabilityReview} \emph{ \cite[Theorem \ref{FIW1-StabilityDegreeImpliesRepStability}]{FIW1}} {\bf (Finitely generated $\FI_{\W}$--modules are uniformly representation stable).} Suppose that $k$ is a characteristic zero field, and that $V$ is an $\FI_{BC}$--module with weight $\leq d$ and stability degree $N$. Then $\{V_n \}$ is uniformly representation stable with respect to the maps $\phi_n: V_n \to V_{n+1}$ induced by the natural inclusions $I_n: \bn \hookrightarrow {\bf (n+1)}$. The sequences stabilizes for $n \geq N+d$.
\end{thm}

Our analysis of stability degree allows us to recasts the bounds in terms of the generation and relation degree. Using properties of induction and restriction of $\FIW$--modules, we can extend these representation stability results to $\FI_D$--modules. 

\begin{thm} \label{FinGenImpliesRepStabilityReview} \emph{ \cite[Theorem \ref{FIW1-FinGenImpliesRepStability}]{FIW1}} {\bf (Finitely generated $\FI_{\W}$--modules are uniformly representation stable).} Suppose that $k$ is a field of characteristic zero, and $\W_n$ is $S_n$, $D_n$ or $B_n$. Let $V$ be a finitely generated $\FI_{\W}$--module of weight $\leq d$, generation degree $\leq g$ and relation degree $\leq r$; when $\W_n$ is $D_n$, take $r$ instead to be an upper bound on the relation degree of $\Ind_{D}^{BC} \; V$.  Then, $\{V_n \}$ is uniformly representation stable with respect to the maps induced by the natural inclusions $I_n: \bn \to {\bf (n+1)}$, stabilizing once $n \geq \max(g,r)+d$; when $\W_n$ is $D_n$ and $d=0$ we need the additional condition that $n \geq g+1$.
\end{thm}

The converse statement follows easily from ``surjectivity'' criterion for representation stability. 

\begin{thm} \label{RepStabilityImpliesFinGenReview} \emph{\cite[Theorem \ref{FIW1-RepStabilityImpliesFinGen}]{FIW1}} {\bf (Uniformly representation stable $\FI_{\W}$--modules are finitely generated).} Suppose conversely that $V$ is an $\FI_{\W}$--module, and that $\{V_n, (I_n)_* \}$ is uniformly representation stable for $n \geq N$. Then $V$ is finitely generated in degree $\leq N$. 
\end{thm}

\subsubsection{Tensor products of $\FIW$--modules and graded $\FIW$--algebras}

\begin{defn} {\bf (Tensor product of $\FIW$--modules).} Given $\FIW$--modules $V$ and $W$, the \emph{tensor product} $V \otimes W$ is the $\FIW$--module by $(V \otimes W)_n = V_n \otimes W_n$ with the diagonal action of the $\FIW$--morphisms. 
\end{defn}

\begin{prop}\emph{\cite[Proposition \ref{FIW1-TensorsPreserveFinGen}]{FIW1}}{\bf(Tensor products respect finite generation).} If $V$ and $W$ are finitely generated $\FIW$--modules, then $V \otimes W$ is finitely generated. If $V$ is generated in degree $\leq m$ and $W$ in degree $\leq m'$, then $V \otimes W$ is generated in degree $\leq m+m'$. Over a field of characteristic zero, the weight of ($V \otimes W$) is at most weight($V$) + weight($W$). 
\end{prop}

\begin{defn}{\bf (Graded $\FIW$--modules; Finite type; Slope).} A \emph{graded $\FIW$--module} $V=\oplus_i V^i$ is a functor from $\FIW$ to the category of graded $k$--modules; each graded piece $V^i$ is an $\FIW$--module. We say $V$ has \emph{finite type} if $V^i$ is a finitely generated for all $i$. 

Suppose $k$ is a field of characteristic zero. Let $V$ be a graded $\FIW$--module over $k$ supported in nonnegative degrees. We define the \emph{slope} of $V$ to be $\leq m$ if for all $i$ the $\FIW$--module $V^i$ has weight $\leq m\cdot i$.

The tensor product of graded $\FIW$--modules $U=\oplus_i U^i$ and $W=\oplus_j W^j$ is the graded $\FIW$--module $$U\otimes W = \bigoplus_{\ell} (U \otimes W)^{\ell} := \bigoplus_{\ell} \bigg( \bigoplus_{i+j=\ell} (U^i \otimes W^j) \bigg). $$

Suppose $U$ and $W$ are supported in nonnegative degrees with $U_0 \cong W_0 \cong M_{\W}({\bf 0})$. If $U$ and $W$ have finite type, then $U \otimes W$ is a graded $\FIW$--module of finite type. Over characteristic zero, if $U$ and $W$ have slopes $\leq m$, then  $U\otimes W$ will have slope $\leq m$. 
\end{defn}

\begin{defn}{\bf($\FIW$--algebras; $\FIW$--ideals; The free associative algebra on $V$).} A \emph{(graded) $\FIW$--algebra} $A = \bigoplus A^i$ is a functor from $\FIW$ to the category of (graded) $k$--algebras. A sub--$\FIW$--module $V$ \emph{generates} $A$ as an $\FIW$--algebra if $V_n$ generates $A_n$ as a $k$--algebra for all $n$.  

 Let $A$ be a graded $\FIW$--algebra. An $\FIW$--ideal $I$ of $A$ is a graded sub--$\FIW$--algebra of $A$ such that $I_n$ is a homogeneous ideal in $A_n$ for each $n$.

We define the \emph{free associative algebra} on  an $\FIW$--module $V$ as the graded $\FIW$--algebra  $$ k \langle V \rangle := \bigoplus_{j=0}^{\infty} V^{\otimes j}.$$ Any $\FIW$--algebra $A$ generated by $V$ admits a surjection of $\FIW$--algebras $k\langle V \rangle \twoheadrightarrow A.$

 If $V$ is supported in nonnegative degrees and has finite type, then so does $A$. Over a field of characteristic zero, if $V$ has slope $\leq m$ then so does $A$. 
 
 \end{defn}
 
 \begin{prop} \label{SlopeFinGenFromGeneratorsReview} \emph{\cite[Proposition \ref{FIW1-SlopeFinGenFromGenerators}]{FIW1}}  Let $A$ be an $\FIW$--algebra generated by a graded $\FIW$--module $V$ concentrated in grade $d$. If $V$ is finitely generated in degree $\leq m$, then the $i^{th}$ graded piece $A^i$ is finitely generated in degree $ \leq \left( \frac{i}{d}\right)m$, and moreover if $k$ is a characteristic zero field then weight$ (A^i) \leq \left( \frac{i}{d}\right)$weight$(V)$. 
\end{prop}

\begin{defn} {\bf (Co--$\FIW$--modules, Co--$\FIW$--algebras, finite type).} A \emph{graded co--$\FIW$--module} is a functor from the dual category $\FIW^{op}$ to the category of graded $k$--modules. A \emph{graded co--$\FIW$--module} is a functor from $\FIW^{\mathrm{op}}$ to graded $k$--algebras. 

A graded co--$\FIW$--module $V$ over a field $k$ has \emph{finite type} if its \emph{dual}  $V^*$, $$V^*_n := \Hom_k(V_n, k),$$ has finite type. Similarly, $V$ has \emph{slope} $\leq m$ if $V^*$ does. 
\end{defn}

\begin{prop}\emph{\cite[Proposition \ref{FIW1-CoModulesGenerateFiniteType}]{FIW1}} {\bf (Finite type co--$\FIW$--modules generate co--$\FIW$--algebras of finite type).}
 Let $k$ be a Noetherian commutative ring. Suppose that $A$ is a graded co--$\FIW$--algebra containing a graded co--$\FIW$--module $V$ supported in positive grades. If $V$ has finite type, then the subalgebra $B$ of $A$ generated by $V$ is a graded co--$\FIW$--algebra of finite type.  When $k$ is a field of characteristic zero and $V$ is a graded co--$\FIW$--module  of slope $\leq m$, then $B$ has slope $\leq m$.
\end{prop}

This concludes the summary of the vocabulary and basic results from \cite{FIW1}.

\section{$\FIW \sh$--modules} \label{SectionFISharp}

Church--Ellenberg--Farb \cite[Definition 2.19]{CEF} define $\FI\sh$--modules, a class of sequences of $S_n$--representations which carry compatible $\FI_A$ and co--$\FI_A$--module structures. We will give several characterizations of the analogous constructions in type B/C.  An $\FI_{BC}\sh$--module structure imposes strong constraints on a sequence of $B_n$--representations; just as for $\FI_A\sh$--modules \cite[Theorem 2.24]{CEF}: the underlying $\FI_{BC}$--module structure on these sequences must be of the form $\bigoplus_{r} M_{BC}(U_r)$ for some set of $B_{r}$--representations $U_r$.

\begin{defn} \label{DefnFISharp} For $n \geq 0 \in \Z$, let $\bno := \{ 0 , \pm 1,\pm 2, \ldots, \pm n \}.$ We think of the digit $0$ as a basepoint. Define $\FI_{BC}\sh$ to be the category with objects $\bno$ for $n \geq 0 \in \Z$, and morphisms 
\begin{align*}
f: \bmo \to \bno \qquad \text{such that }  & \qquad   f(-a) = -f(a) \text{ for all $a \in \bno$} \qquad \\  \text{ and } & \qquad | f^{-1}(b)| \leq 1 \text{ for $ 1 \leq |b| \leq n$}.
\end{align*}
These morphisms are ``injective away from zero''. Note that the conditions imply $f(0)=0$.  

We define $\FI_A\sh$ to be the subcategory with the same objects and morphisms preserving signs. In both cases, an \emph{$\FIW\sh$--module} over a ring $k$ is a functor from $\FIW$ to the category of $k$--modules.

In both types A and B/C, the injective maps in $\FIW\sh$ are precisely the $\FIW$ morphisms. We call $f \in \Hom_{\FIW\sh}(\bmo, \bno)$ a \emph{projection} if $| f^{-1}(\pm b)| = 1 \text{ for $ 1 \leq b \leq n$}$; these projections are left inverses to the $\FIW$ morphisms. 

For a morphism $f:\bmo \to \bno$, we call $|f^{-1}(\{\pm 1 , \ldots, \pm n\})|$ the \emph{rank} of $f$. 
 \end{defn}

This description of $\FIW\sh$--modules was suggested to us by Peter May. The category $\FI_A\sh$ appears independently in work by May and Merling studying the Segal equivariant infinite loop space machine \cite{MayMerling}, where the category is denoted $\Pi$. 

\begin{rem}{\bf (A Category $\FI_D\sh$?).} \label{RemarkFIDSharp} Unlike with $\FI_A$ and $\FI_{BC}$, we cannot introduce partial inverses to the morphisms in the category $\FI_{D}$ without also introducing additional automorphisms -- and, in fact, generating the entire category $\FI_{BC}\sh$. It is not clear that we can create any satisfactory analogue of $\FI\sh$--module theory in type D, since critical properties fail: the $\FI_D$--module structure on $M_D(\bm)$ does not extend to an $\FI_{BC}\sh$--module structure.
\end{rem}

\subsubsection*{An alternate characterizations of the categories $\FI_{\W}\sh$}

In this section we relate the categories $\FIW\sh$ back to the definition of $\FI\sh$ given by Church--Ellenberg--Farb \cite[Definition 2.19]{CEF}. 

\begin{rem} \label{FISharpAsTriples} Church--Ellenberg--Farb \cite[Definition 2.19]{CEF} defined $\FI_{A} \sh$ to be the category whose objects are finite sets, in which $\Hom_{\FI_A \sh}(S, T )$ is the set of triples $(A, B, \phi)$ with $A$ a subset of $S $, $B$ a subset of $T$ and $\phi : A \to B$ an isomorphism. The composition of two morphisms is given by composition of functions, where the domain is the largest set on which the composition is defined, and the codomain is its bijective image.  

We can generalize this definition. We call a subset $A \subseteq \bn$ \emph{symmetric} if $a \in A \Longleftrightarrow -a \in A.$

Then we define $\FI_{BC}\sh$ to be the category whose objects are the finite sets $\bn = \{ \pm 1, \pm 2, \ldots, \pm n \},$ and whose morphisms $\Hom(\bm, \bn)$ are triples $(A, B, \phi)$ such that $A$ is a symmetric subset of $\bm$, $B$ is a symmetric subset of $\bn$, and $\phi:A \to B$ is an injective map satisfying $\phi(-a) = -\phi(a)$ for every $a \in A.$ 

The subcategory $\FI_A \sh$ comprises all morphisms that preserve signs; this coincides with the definition of $\FI_A \sh$ given in \cite[Definition 2.19]{CEF}. 

These definitions of the categories $\FI_{A}\sh$ and $\FI_{BC}\sh$ are equivalent to Definition \ref{DefnFISharp}. We identify the morphism $(A, B, \phi) \in \Hom_{\FIW \sh}(\bm, \bn )$ with the map $f: \bmo \to \bno$ defined by
\begin{align*}
f &: \bmo \to \bno \\
j & \mapsto \left\{ \begin{array}{ll}
         \phi(j) & \mbox{$ j \in A  $},\\
        0 & \mbox{$ j \notin A.$}.\end{array} \right.
\end{align*}

Conversely, we can identify $f: \bmo \to \bno$ with a triple $(A, B, \phi)$ by taking $A = f^{-1}(\{ \pm 1, \ldots, \pm n \})$, $B=f(A)$, and $\phi = f \vert _A$. One can check that these identifications of morphisms are consistent with the composition rules, and give an isomorphism of categories. 
\end{rem}

\subsubsection*{Examples of $\FI_{\W}\sh$--modules}

We prove in Proposition \ref{M(m)IsFISharp} and Corollary \ref{M(U)isFISharp} that $\FI_{BC}$--modules of the form $M_{BC}(\bm)$ or $M_{BC}(U)$ have $\FI_{BC}\sh$--module structures. 

\begin{prop} {\bf ($M_{\W}(\ba)$ is an $\FI_{\W}\sh$--module).}\label{M(m)IsFISharp} Let $\W_n$ be $S_n$ or $B_n$. For $a \geq 0$, the $\FIW$--module structure on $M_{\W}(\ba)$ extends to an $\FIW\sh$--structure. 
\end{prop}

\begin{proof}[Proof of Proposition \ref{M(m)IsFISharp}]
 By Definition \ref{Defn:M(m)Review}, $M_{\W}(\ba)_{m} = \Span_k \left\{e_s \mid s \in \Hom_{\FIW}(\ba, \bm)\right\}.$
 
Take any $\FIW\sh$--morphism $f:\bmo \to \bno$. We define
\begin{align*}
 f \cdot e_s =  &  \left\{ \begin{array}{ll}
         e_{f \circ s}  & \mbox{if $ 0 \notin f( s(\ba))$ },\\
         0 & \mbox{otherwise}.\end{array} \right.
\end{align*}

The condition $ 0 \notin f( s( \ba))$ is the statement that $f \circ s$ is an injective map $\ba \to \bn$. Given $g: \bno \to \bpo$, we note that $$ 0 \in \Big(g \circ f\Big)\Big( s(\ba)\Big) \quad \Longleftrightarrow \quad 0 \in g\Big((f \circ s)(\ba)\Big); $$ this implies that the action $(g \circ f) \cdot e_s = g \cdot ( f \cdot e_s)$ is functorial. 
\end{proof}

\begin{cor} {\bf ($M_{\W}(U)$ is an $\FIW\sh$--module).} \label{M(U)isFISharp} Let $\W_n$ be $S_n$ or $B_n$. Given a $\W_a$--representation $U$, the $\FIW$--module $$M_{\W}(U) := M_{\W}(\ba) \otimes_{k[\W_a]} U$$ has the structure of an $\FIW\sh$--module. \end{cor}

\begin{rem} We note the proof of Proposition \ref{M(m)IsFISharp} does not work in type $D$, as the space $M_D(\bm)_m \subseteq M_{BC}(\bm)_m$ is not closed under the action of action of the $\FI_{BC}\sh$--morphisms.
\end{rem}

\subsection{Classification of $\FI_{BC}\sh$--modules}

The structure of an $\FI_{BC}\sh$--module is highly constrained. In Corollary \ref{M(U)isFISharp} we saw that $M_{BC}(U)$ is an $\FI_{BC}\sh$--module. Just as Church--Ellenberg--Farb proved with $\FI_A\sh$--modules \cite[Theorem 2.24]{CEF}, we will now find that all $\FI_{BC}\sh$--modules are sums of $\FI_{BC}\sh$--modules of this form. For the following theorem, recall the definition of $\text{FB$_\W$--Mod}$ from Definition \ref{Defn:M(m)Review}. 

\begin{thm}\label{ClassifyingFISharp} {\bf ($\FI_{\W} \sh$--modules take the form $\bigoplus_{a=0}^{\infty} M_{\W}(U_a)$).}
 Let $\W_n$ be $S_n$ or $B_n$. Every $\FIW \sh$--module $V$ is of the form $$V= \bigoplus_{a=0}^{\infty} M_{\W}(U_a), \qquad \text{ $U_a$ a representation of $\W_a$ (possibly $U_a=0$),}$$ and moreover that the maps 
$$M_{\W}( - ) : \text{FB$_\W$--Mod}  \longrightarrow \FIW \text{--Mod} \qquad \text{and} \qquad H_0(-) :  \FIW \text{--Mod}  \longrightarrow \text{FB$_\W$--Mod} $$ 
define an equivalence of categories.
\end{thm}

We remark that Theorem \ref{ClassifyingFISharp}, and in particular the result in type A proven by Church--Ellenberg-Farb \cite[Theorem 2.24]{CEF}, parallels earlier work of Pirashvili \cite{PirashviliGamma, PirashviliHodge} on \emph{$\Gamma$--modules}, functors from the category of finite based sets and (not necessarily injective) based maps.

This theorem is proved for $\FI_A$ in \cite[Theorem 2.24]{CEF}, and their proof adapts readily to the general case. 

Church--Ellenberg--Farb proceed by induction. Assume $V$ is an $\FI_A \sh$--module, and fix $n$ such that $V_m = 0$ for all $m<n$ (possibly $n=1$). They define a particular idempotent endomorphism of $\FI_A\sh$--modules $E: V \to V$, and prove the resultant decomposition $$V = EV \oplus \ker(E) \cong M_A(V_n) \oplus \ker(E). $$ 

Their same proof carries through exactly in the case $\FI_{BC} \sh$--modules if we redefine the endomorphism $E$ as follows, for $m \geq n$,
\begin{align*}
E_{m} \; &: V_m \to V_m \\
E_{m} \; & = \sum_{\substack{S \subseteq \bm, \; |S|=n \\ S \text{ symmetric}}} I_S \qquad \qquad  \text{where} \quad &I_S(j) =   \left\{ \begin{array}{ll}
         j  & \mbox{if $j \in S$  },\\
         0 & \mbox{otherwise}\end{array} \right. \\ && \in \Hom_{\FI_{BC} \sh}(\bmo, \bmo).  
\end{align*}
In the notation of Remark \ref{FISharpAsTriples}, $I_S = (S,S, \text{identity}).$
Again we conclude $$V \cong M_{BC}(V_n) \oplus \ker(E) $$ with $\ker(E)$ vanishing in degree $n$, and the desired decomposition follows by induction on $n$. 

Church--Ellenberg--Farb further argue that, since maps $F: V \to V'$ of $\FI_A \sh$--modules commute with $E$ and preserves this decomposition, the map $M_A(V_n) \to M_A(V'_n)$ must be induced by some map of $S_n$--representations $V_n \to V_n'$. Their arguments hold for $\FI_{BC}$--modules, and imply the equivalence of the categories $\text{FB$_{BC}$--Mod}$ and $\FI_{BC}\sh \text{--Mod}$. 

\begin{cor} \label{SharpWeightd} Let $\W_n$ be $S_n$ or $B_n$. With $V$ an $\FIW\sh$--module as above, any sub--$\FIW\sh$--module of $V$ is of the form $\bigoplus_{a=0}^{\infty} M_{\W}(U_a')$ for some $\W_a$--representations $U_a' \subseteq U_a$ (possibly zero). 
\end{cor}

\begin{cor}
 Let $\W_n$ be $S_n$ or $B_n$. If $V$ is an $\FIW$--module generated in degree $\leq d$, then any sub--$\FIW\sh$--module of $V$ is also generated in degree $\leq d$. 
\end{cor}

\begin{cor} \label{SharpStabilityDegd} Let $\W_n$ be $S_n$ or $B_n$. An $\FIW\sh$--module $V$ has injectivity degree $0$. If $V$ is generated in degree $\leq d$, then $V$ has stability degree $\leq d$, as do its sub--$\FIW$--submodules.
\end{cor}
\begin{proof}[Proof of Corollary \ref{SharpStabilityDegd}]
 Let $U_m$ be any $\W_m$--representation. In \cite[Example \ref{FIW1-M(m)FinGen}]{FIW1} we saw that $M_{\W}(U_m)$ is generated in degree $m$. In \cite[Proposition \ref{FIW1-StabilityDegreeM(U)}]{FIW1} we further saw that $M_{\W}(U_m)$ has injectivity degree $0$ and surjectivity degree $m$. By assumption that $V$ is generated in degree $\leq d$, we may decompose $V$ into a sum $  V=\bigoplus_{m=0}^d M_{\W}(U_m)$ and the result follows. 
\end{proof}

\begin{cor} \label{SharpRepStableDeg2d}
 If $V$ is an $\FI_{BC}\sh$--module over characteristic zero, generated in degree $d$. Then $\{V_n\}$ is uniformly representation stable in degree $\leq 2d$. 
\end{cor}

\begin{proof}[Proof of Corollary \ref{SharpRepStableDeg2d}]
 Any such $\FI_{BC}\sh$--module has weight $ \leq d$ by Theorem \ref{WnDiagramSizesReview}, and stability degree $\leq d$ by Corollary \ref{SharpStabilityDegd}. The conclusion follows from Theorem \ref{StabilityDegreeImpliesRepStabilityReview}.
\end{proof}

\begin{cor} \label{FISharpDeterminedByCharacters} Let $\W_n$ be $S_n$ or $B_n$. An $\FIW\sh$--module $V$ is completely determined by the sequence of $\W_n$--representations $\{ V_n \}$, equivalently (over characteristic zero) by the sequence of characters $\{ \chi_n \}$. \end{cor}

\begin{proof}[Proof of Corollary \ref{FISharpDeterminedByCharacters}]
 We can construct $H_0(V)$ inductively from the sequence $\{V_n\}$ of $\W_n$--representations:
$$ H_0(V)_0 = V_0 \qquad \text{and} \qquad H_0(V)_n =  V_n \; / \; \text{span} \coprod_{k<n}  M_{\W}\Big(H_0(V)_k\Big)_n $$
The $\FIW\sh$--module structure on $V$ is determined by the identification $V \cong M_{\W}(H_0(V)). \hfill \qedhere$ \end{proof}

\begin{cor} \label{DimensionFISharp}
 If $k$ is a field, and $V$ an $\FI_{BC} \sh$--module $k$. Then 
 \begin{align*}
  V  \text{ is finitely generated in degree $\leq d$}  
 &  \Longleftrightarrow \text{  $\dim_k(V_n) = O(n^d)$ }  \\
  & \Longleftrightarrow \text{  $\dim_k(V_n) = P(n)$ for a polynomial $P \in \Q[T]$ of degree at most $d$} 
 \end{align*}
 If $k$ is a commutative ring, then an $\FI_{BC} \sh$--module $V$ over $k$ is finitely generated in degree $\leq d$ if and only if $V_n$ is generated by $O(n^d)$ elements. 
\end{cor}

\begin{proof}[Proof of Corollary \ref{DimensionFISharp}]
The statements follow from Theorem \ref{ClassifyingFISharp} and the same argument used to prove \cite[Corollary 2.27]{CEF}. In type B/C, the polynomial $P$ is determined by the formula
$$ \dim_k M_{BC}(U)_n = { n \choose m}\dim_k U  \qquad \text{ for a $B_m$--representation $U$.} \qedhere $$ 
\end{proof}

\section{The character polynomials }  \label{SectionCharPolys}

\subsection{Character polynomials for the symmetric groups}

A major result of Church--Ellenberg--Farb is that, given a finitely generated $\FI_A$--module over a field of characteristic zero, the characters of the $S_n$--representations $V_n$ have a particularly nice form. They are, for $n$ sufficiently large, given by a character polynomial (independent of $n$), as we now define. 

\begin{defn} {\bf (Character Polynomials for $S_n$).} Let $k$ be a characteristic zero field. For $r \geq 1$ and $n \geq 0$, let $X_r$ be the class function on $S_n$ defined by 
$$X_r(s) := \text{ the number of $r$--cycles in the cycle type of $s$}.$$

For fixed $n$, the monomials in the functions $X_1, \ldots, X_n$ span the space of class functions on $S_n$, subject to some relations -- for example, relations imposed by the fact that any element's cycle lengths sum to $n$. As functions on the disjoint union $\coprod_{n=0}^{\infty} S_n$, however, the functions $X_r$ are algebraically independent, and define a polynomial ring $k[X_1, X_2, \ldots, ]$. We call elements of this ring the \emph{character polynomials} of the symmetric groups, and define the total degree of a character polynomial by assigning $\deg(X_r) = r$.  \end{defn}

\begin{thm}\cite[Theorem 2.67]{CEF} {\bf (Polynomiality of characters for $S_n$).} \label{SnPersinomial}
Let $k$ be a field of characteristic zero, and let $V$ be a finitely generated $\FI_A$--module with weight $\leq d$ and stability degree $\leq s$. There exists a unique polynomial $P_V \in k[X_1 , X_2 , \ldots]$ such that 
$$  \chi_{V_n} (\sigma) = P_V (\s) \qquad \text{for all $n \geq s + d$ and all $\s \in S_n$}. $$                                                                                                         
The polynomial $P_V$ has degree at most $d$. By setting $F_V(n) = P_V (n, 0, . . . , 0)$ we have:
$$  \dim_k(V_n) = \chi_{V_n}(id) = F_V (n) \qquad \text{ for all $n \geq s + d$}.$$
If $V$ is an $\FI_A\sh$--module then the above equalities hold for all $n \geq 0$. 
\end{thm}

{ \noindent \bf  Background and formulas for character polynomials of the symmetric groups.  \quad} Church--Ellenberg--Farb \cite{CEF} prove these theorems using the classical result that the character of the irreducible representations $V(\y)_n$, written here in the notation defined in Section \ref{SectionRepTheoryReview}, is given by a character polynomial $P^{\y}$ independent of $n$. These character polynomials were described by Murnaghan in 1951 \cite{MurnaghanCharacterPolynomials} and by Specht in 1960 \cite{SpechtCharaktere}. This independence of the characters from $n$ was known to Murnaghan in 1937 \cite{MurnaghanCharacters}.

Formulas for the character polynomial $P^{\y}$ associated to the irreducible representations $V(\y)_n$ are given in Macdonald's book \cite{MacdonaldSymmetric}. In 2009, new formulas were published by Garsia and Goupil \cite{GarsiaGoupil}, which they used to study the combinatorics of Kronecker coefficients. To state these formulas, we use the following notation: 

Let $\y$ be a partition of $n$. We define the \emph{length} of $\y$, $\ell(\y) := \text{ the number of parts of $\y$.}$

 For $ r \geq 1$, we write $n_r(\y) := \text{ the number of parts of $\y$ of length $r$.}$

We further define the integer $z_{\y}$ so that $\displaystyle \frac{n!}{ z_{\y}}$ is the number of elements in $S_n$ of cycle type $\y$.  Explicitly,  $$z_{\y} = \prod_r^n r^{n_r(\y)} n_r(\y)!$$

We write $\chi^{\y}$ to mean the character of the irreducible $S_n$--representation $V_{\y}$. We define $\chi^{\varnothing} :=1$. If $\chi$ is any class function on $S_n$, and $\rho$ a partition of $n$, then we write $\chi_{\rho}$ to mean the value of $\chi$ on elements of cycle type $\rho$. 

\begin{defn} {\bf (Generalized Binomial Coefficients).}
 Let $\rho$ be a partition of $m$. Following Macdonald \cite[I.7.13(a)]{MacdonaldSymmetric}, we define \emph{generalized binomial coefficients}: 
$$\binom{\bX}{\rho} := \prod_r \binom{X_r}{n_r(\rho)} = \prod_r \frac{X_r(X_r -1) \cdots (X_r - n_r(\rho) +1)}{n_r(\rho)!}, $$ 
 For example, $$\binom{\bX}{(3,2,2,1)} := \binom{X_3}{1} \binom{X_2}{2} \binom{X_1}{1} =  X_3 \frac{X_2(X_2-1)}{2}X_1 = \frac{1}{2}X_3X_2^2X_1-\frac{1}{2}X_3X_2X_1$$
\end{defn}

\begin{rem} {\bf (Indicator Functions for the Conjugacy Classes of $S_m$) } \label{IndicatorSn}
Given a partition $\y$ of $m$, and $s \in S_m$, note that the generalized binomial coefficient
\begin{align*}
\binom{\bX}{\y} (s) &= \left\{ \begin{array}{ll}
         1 & \mbox{if $s$ has cycle type $\y$},\\
         0 & \mbox{otherwise}.\end{array} \right.
\end{align*}
This polynomial's restriction to $S_m \subseteq \coprod_{n \geq 0} S_n$ is an indicator function for the conjugacy class of cycle type $\y$. Polynomials of this form give a convenient basis for $k[X_1, X_2, \ldots, ]$. \\ 

Since the binomial coefficient in an indeterminate $X$ 
$${X \choose m} = \frac{X(X-1)(X-2) \cdots (X-m+1)}{m !}$$
 is a polynomial in $X$ of degree $m$, the generalized binomial coefficient $\displaystyle \binom{\bX}{\y} $ is a polynomial of total degree $\sum r \cdot n_r(\y) = m$ in $k[X_1, X_2, \ldots ]$.
\end{rem}

\begin{prop}(\cite[I.7.14]{MacdonaldSymmetric})  \label{MacDonaldFormula1} 

 For $\y \vdash m$, a formula for the character $P^{\y}$ of the irreducible $S_n$--representation $V(\y)_n$ is given as follows: 
$$ P^{\y} = \sum_{\substack{ \text{Partitions } \rho, \sigma \\ |\rho| + |\s| = |\y| }} \frac{(-1)^{\ell(\s)} \; \chi^{\y}_{(\rho \cup \s)} }{z_{\s}} \;   \;  \binom{\bX}{\rho}.$$ 

\end{prop}

By Remark \ref{IndicatorSn}, this is a character polynomial of degree $|\y|=m$.

\subsection{Character polynomials in type B/C and D} \label{SectionCharPolysBCD}

We can analogously define character polynomials for the hyperoctahedral group $B_n$ and the even-signed permutation groups $D_n$.  Our goal is to prove that the characters of a finitely generated $\FIW$--module are equal (for $n$ sufficiently large) to a character polynomial. A summary of the main result:

\begin{thm} \label{CharPolySummary} {\bf (Summary: Finitely generated $\FIW$--modules have character polynomials).} Let $V$ be an $\FIW$--module over characteristic zero finitely generated in degree $\leq d$. Let $\chi_{V_n}$ denote the character of the $\W_n$--representation $V_n$. Then there exists a unique character polynomial $F_V$ of degree at most $d$ such that $F_V(\sigma) = \chi_{V_n}(\sigma)$ for all $\sigma \in \W_n$, for all $n$ sufficiently large. 
\end{thm}

\noindent The full statement of this result in type B/C and D, including bounds on the stable range, is given in Theorem \ref{WnPersinomial} in Section \ref{SectionCharPolysBCD}. To this end, we will first develop the theory in type B/C, and from there we can use our methods of inducing $\FI_D$--modules to $\FI_{BC}$ to recover results in type D.

 Recall from Section \ref{RepTheoryBnReview} that the conjugacy classes for $B_n$ are classified by double partitions $(\a, \b)$ of $n$, designating the signed cycle type of each element. Given a character (or class function) $\chi$ of a $B_n$--representation, we will write $\chi_{(\a,\b)}$ to denote the value of $\chi$ on elements of signed cycle type $(\a,\b)$.

\begin{defn}
 {\bf (Character Polynomials for $B_n$).} For $r \geq 1$ and $n \geq 0$, let $X_r$ and $Y_r$ be the class functions on $B_n$ defined by 
$$X_r(w) = \text{ the number of positive $r$--cycles in the cycle type of $w$}.$$ 
$$Y_r(w) = \text{ the number of negative $r$--cycles in the cycle type of $w$}.$$ 
Again, these functions form a polynomial ring $k[X_1, Y_1, X_2, Y_2,\ldots]$ where we designate $deg(X_r) = deg(Y_r) = r$. 
\end{defn}

\begin{ex}
 We saw in Example \ref{Example:SignedPermMatrices} that $V_n = V((n-1), (1)) \cong k^n,$ the canonical $B_n$--representation by signed permutation matrices, has characters $\chi^{V} = X_1 - Y_1$ for all $n$. Similarly, one can compute that the characters of $V_n = \bigwedge^2 \, V\big((n-1), (1)\big) $ are 
$$\chi^{\bigwedge^2 V} = \frac12X_1(X_1-1) + \frac12 Y_1(Y_1-1) -X_1Y_1 - X_2 + Y_2 \qquad \text{for all $n$},$$ 
and that the characters of $V_n = $ Sym$^2 \, V\big((n-1), (1)\big) $ are 
$$\chi^{\text{Sym}^2 V} = \frac12X_1(X_1+1) + \frac12Y_1(Y_1+1) -X_1Y_1 + X_2 - Y_2 \qquad \text{for all $n$}.$$
\end{ex}

\begin{rem} {\bf (Indicator Functions for the Conjugacy Classes of $B_m$) } \label{GeneralizedBinomialCoeff}
Given a double partition $(\y, \nu)$ of $m$, and $w \in B_m$, note that the degree $m$ character polynomial 
\begin{align*}
\binom{\bX}{\y}\binom{\bY}{\nu} (w) &= \left\{ \begin{array}{ll}
         1 & \mbox{if $w$ has signed cycle type $(\y, \nu)$},\\
         0 & \mbox{otherwise}.\end{array} \right.
\end{align*}
Again $\displaystyle \binom{\bX}{\y}\binom{\bY}{\nu} $ is a polynomial of degree $\sum r \cdot n_r(\y) + \sum r \cdot n_r(\nu) = m$ that is an indicator function on $B_m$ of the signed conjugacy class $(\y, \nu)$. 
\end{rem}

\begin{rem} {\bf (Restricting Characters to $S_n \subseteq B_n$).} \label{RestrictingCharactersWntoSn}
 The symmetric group $S_n$ forms the subgroup of $B_n$ generated by the (necessarily positive) cycles that preserve signs. Thus, if $V$ is a $B_n$--representations with character $\chi^V$ given by some character polynomial $P_V \in  k[X_1, Y_1, X_2, Y_2, \ldots]$, the character for $\Res_{S_n}^{B_n} V$ is given by the element in $k[X_1, X_2, \ldots]$ obtained by evaluating each variable $Y_r$ in $P_V$ at $0$. 
\end{rem}

\subsubsection{The character of $V(\y, \mu)_n$ is independent of $n$}

Recall from Section \ref{RepTheoryBnReview} that, given a double partition $(\y, \nu)$ of $d$ with $\nu \vdash m$,  then $V(\y, \nu)_n$ denotes the irreducible $B_n$--representation associated to the double partition $(\y[n-m], \nu)$. 

\begin{thm} \label{ThmWnIndependentCharacters} {\bf(The character of $V(\y, \mu)_n$ is independent of $n$).}
 If $(\y, \nu)$ is a double partition of $d$, then there is a character polynomial $P^{(\y, \nu)}$ of degree at most $d$ equal to the character of the irreducible $B_n$--representations $V(\y, \nu)_n$ for all $n$. 

Explicitly, $P^{(\y, \nu)}$ is given as follows. Let $m=|\nu|$, and define $\mu$ so that $\nu = \mu[m]$; for $\nu= \varnothing$ take $\mu = \varnothing$. Then  
\begin{align*} P^{(\y, \nu)} =  & \sum_{\substack{(\a, \b) \\ |\a| + |\b| = |\nu|, }} \; \; \sum_{\substack{ \text{Partitions } \rho, \sigma \\ |\rho| + |\s| = |\mu| }} \; \;   \sum_{\substack{ \text{Partitions } \xi, \eta \\ |\xi| + |\eta| = |\y| }} \; \; (-1)^{\ell(\b)} \Bigg( \frac{(-1)^{\ell(\s)} \; \chi^{\mu}_{(\rho \cup \s)} }{z_{\s}} \Bigg) \Bigg(  \frac{(-1)^{\ell(\eta)} \; \chi^{\y}_{(\xi \cup \eta)} }{z_{\eta}}  \Bigg) \\ 
& \; \Bigg(  \;   \;  \prod_{r}  \binom{ n_r(\a) + n_r(\b)}{n_r(\rho)} \binom{ X_r - n_r(\a) + Y_r - n_r(\b)}{n_r(\xi)} \binom{ X_r}{n_r(\a) } \binom{ Y_r} {n_r(\b) } \Bigg).
\end{align*}
\end{thm}

For example, 
\begin{align*}
P^{\Big( \Y{1} \;,\; \Y{1} \Big)} &= (X_1 - Y_1)(X_1 + Y_1 -2) \qquad \text{and} \qquad
P^{\Big( \varnothing \;,\; \Y{1,1} \Big)} = { X_1 \choose 2 }+ { Y_1 \choose 2 } - X_2 -X_1Y_1 +Y_2  
\end{align*}

We will prove Theorem \ref{ThmWnIndependentCharacters} in four steps. Our first step, Lemma \ref{LemPullbackCharacters}, is to prove the result for representations of the form $V(\y, \varnothing)_n$. In the second step, Lemma \ref{LemSignCharacters}, we produce a formula for characters of representations $V( \varnothing, \y[n])$. Our third step, Lemma \ref{LemInducedRepFormula}, is to compute the character of an induced representation of the form $\Ind_{B_m \times B_{n-m}}^{B_n} U \boxtimes U'$, and the final step will be to derive the formula in Theorem \ref{ThmWnIndependentCharacters}.

\begin{lem} \label{LemPullbackCharacters} {\bf (Step 1: The character of $V(\y, \varnothing)_n$).}
 Let $\y$ be a partition of $m$. Then, for each $n$, the character of $B_n$--representation $V(\y, \varnothing)_n$ is given by the character polynomial $P^{(\y,0)}$ 
 
\begin{align*} P^{(\y,0)} & = \sum_{\substack{ \text{Partitions } \rho, \sigma \\ |\rho| + |\s| = |\y| }} \frac{(-1)^{\ell(\s)} \; \chi^{\y}_{(\rho \cup \s)} }{z_{\s}} \;   \;  \binom{\bX + \bY}{\rho} \\
& := \sum_{\substack{ \text{Partitions } \rho, \sigma \\ |\rho| + |\s| = |\y| }} \frac{(-1)^{\ell(\s)} \; \chi^{\y}_{(\rho \cup \s)} }{z_{\s}} \;   \;  \prod_r \binom{X_r + Y_r}{n_r(\rho)}
\end{align*}

\end{lem}

\begin{proof}[Proof of Lemma \ref{LemPullbackCharacters}]
As described in Section \ref{RepTheoryBnReview}, the $B_n$--representations $V(\y, \varnothing)_n$ are by definition the pullback of the $S_n$--representation $V(\y)_n$ under the natural surjection $B_n \twoheadrightarrow S_n$. This map takes positive and negative $r$-cycles in $B_n$ to $r$-cycles in $S_n$; a signed permutation of signed cycle type $(\mu, \nu)$ is mapped to a permutation of type $\mu \cup \nu$. It follows that a hyperoctahedral character polynomial for $V(\y, \varnothing)_n$ can be obtained from the symmetric character polynomial for $V(\y)_n$ by replacing each $X_r$ with the sum $X_r + Y_r$. The formula therefore follows from Macdonald's formula, Proposition \ref{MacDonaldFormula1}. 
\end{proof}

\begin{lem}\label{LemSignCharacters} {\bf (Step 2: The character of $V(\varnothing, \y[n])$).}
 Let $n$ be fixed, and consider a partition $\y[n]$ of $n$. Then the character $\chi^{(\varnothing, \y[n])}$ of the $B_n$--representation $V(\varnothing, \y[n])$ takes the following value on $B_n$ elements of signed cycle type $(\a, \b)$: 

$$ \chi^{(\varnothing, \y[n])}_{(\a, \b)} = (-1)^{\ell(\b)} P^{(\y,0)}(\a, \b). $$ 
\end{lem}

\begin{rem} We note that this formula for the character $V(\varnothing, \y[n])$ is not a $B_n$ character polynomial, since the coefficient $(-1)^{\ell(\b)}$ depends on the cycle type $(\a,\b)$. \end{rem}

\begin{proof}[Proof of Lemma \ref{LemSignCharacters}]
 Recall from Section \ref{RepTheoryBnReview} that $\varepsilon : B_n \to B_n/D_n \cong \{ \pm 1 \}$ is the character mapping an element $w \in B_n$ to $-1$ precisely when $w$ reverses an odd number of signs. Since positive cycles reverse an even number of signs, and negative cycles reverse an odd number, the character $\varepsilon$ takes the value $(-1)^{\ell(\b)}$ on elements of signed cycle type $(\a, \b)$.  

By definition, $$V(\varnothing, \y[n]) = V(\y[n], \varnothing) \otimes \varepsilon =  V(\y, \varnothing)_n \otimes \varepsilon $$
and so the formula follows from Lemma \ref{LemSignCharacters}. 
\end{proof}

\begin{lem} \label{LemInducedRepFormula} {\bf (Step 3: The character of $\Ind_{B_m \times B_{n-m}}^{B_n} U \boxtimes U'$).}
Suppose that $U$ is a $B_m$--representation with character $\chi^U$, and that $U'$ is a $B_{n-m}$--representation, with character $\chi^{U'}$. Then the character $\chi^{(U, U')}$ of the induced $B_n$--representation $\Ind_{B_m \times B_{n-m}}^{B_n} U \boxtimes U'$ is given by:

$$ \chi^{(U, U')}_{(\rho, \s)} = \Bigg(\sum_{\substack{(\a, \b) \\ |\a| + |\b| = m, }} \chi^{U}_{(\a, \b)} \; \chi^{U'}_{(\d, \g)} \; \binom{ \bX}{\a } \binom{ \bY} {\b } \Bigg)(\rho, \s) $$

where $(\d, \g)$ is the double partition of $(n-m)$ such that $(\rho, \s)=(\a \cup \d, \b \cup \g)$. It is well-defined, since $\displaystyle \Bigg( \binom{ \bX}{\a } \binom{ \bY} {\b } \Bigg)(\rho, \s)$ will vanish unless such a decomposition of $(\rho, \s)$ exists. 
\end{lem}

We note that Lemma \ref{LemInducedRepFormula} holds when $k$ is $\Z$ or any field.

\begin{proof}[Proof of Lemma \ref{LemInducedRepFormula}] 

 Let $w \in (B_{m} \times B_{n-m})$, and let $p_{m}$ and $p_{n-m}$ denote the projections of $w$ onto $B_m$ and $B_{n-m}$, respectively. The character of the $(B_m \times B_{n-m})$--representation $U \boxtimes U'$ is 
$$ \chi^{U \boxtimes U'} = \chi^{U}(p_m(w)) \cdot \chi^{U'}(p_{n-m}(w)).$$
 
The character of the induced representation $\Ind^{B_n}_{B_{m} \times B_{n-m}} U\boxtimes U'$ is 
 \begin{align*} \chi^{(U,U')} (w) &= \sum_{\substack {\{ \text{cosets } C \; | \; w \cdot C = C \} \\ \text{ any } s \in C}} \chi^{U \boxtimes U'}(s^{-1}ws) 
\\ &= \sum_{\substack { \{ \text{cosets } C \; | \; w \cdot C = C \} \\ \text{ any } s \in C}} \chi^{U}(p_m(s^{-1}ws)) \cdot \chi^{U'}(p_{n-m}(s^{-1}ws)) \end{align*}
summed over all cosets $C$ in $B_n / (B_{m} \times B_{n-m})$ that are stabilized by $w$, equivalently, those cosets $C$ such that $s^{-1} w s \in (B_{m} \times B_{n-m})$ for any $s \in C$.  

The cosets $B_n / (B_{m} \times B_{n-m})$ correspond to the orbit of the sets $$\{ \{ -1, 1\}, \ldots, \{-m,m\} \} \qquad \text{and} \qquad \{ \{ -(m+1), (m+1)\}, \ldots, \{-n, n\} \}$$ under the action of $B_n$; they are indexed by all partitions of $\{ \{-1,1\}, \{-2, 2\}, \ldots, \{-n, n\} \}$ into a set of $m$ blocks and a set of $(n-m)$ blocks. 

An element $w \in B_n$ can be conjugated into $(B_{m} \times B_{n-m})$ precisely when its positive and negative cycles can be partitioned into a set of cycles of total length $m$, and a set of cycles of total length $(n-m)$. If we fix a double partition $(\a,\b)$ of $m$, then the cycles of $w$ can be factored into an element $w_m$ of cycle type $(\a, \b)$ and its complement $w_{n-m}$ in the following number of ways (possibly $0$):
$$\binom{ \bX}{\a } \binom{ \bY} {\b }(w) := \binom{ X_1(w)}{n_1(\a) } \binom{ X_2(w)} {n_2(\a) } \cdots  \binom{ X_m(w)}{n_m(\a) } \binom{ Y_1(w)}{n_1(\b) } \binom{ Y_2(w)}{n_2(\b) } \cdots \binom{ Y_m(w)} {n_m(\b) }.$$ 
Each such factorization of $w$ corresponds to a coset $C \in B_n / (B_{m} \times B_{n-m})$ that is stabilized by $w$. For any representative $s \in C$, $p_m(s^{-1}ws)$ has signed cycle type $(\a,\b)$.

Thus, if we denote the signed cycle type of $w_{n-m}$ by $(\d, \g)$, we conclude
$$ \chi^{(U, U')}(w) = \Bigg(\sum_{\substack{(\a, \b) \\ |\a| + |\b| = m, }} \chi^{U}_{(\a, \b)} \; \chi^{U'}_{(\d, \g)} \; \binom{ \bX}{\a } \binom{ \bY} {\b } \Bigg)(w). \qedhere $$ \end{proof}

\begin{proof}[Proof of Theorem \ref{ThmWnIndependentCharacters}] {\bf(Step 4: The Character of $V(\y, \nu)_n$).}
 Let $(\y, \nu)$ be a double partition of $d$, with $|\nu|=m$ and $|\y|=(d-m)$. From the construction of the irreducible representations of $B_n$ described in Section \ref{RepTheoryBnReview}, 
$$ V(\y, \nu)_n = \Ind_{B_{n-m} \times B_m}^{B_n} V(\y, \varnothing)_{n-m} \boxtimes V(\varnothing, \nu).$$ 
\noindent We wish to compute a character polynomial $P^{(\y, \nu)}$  which gives the character for $V(\y, \nu)_n$ for each $n$.  

By Lemma \ref{LemInducedRepFormula}, 
$$\chi^{(\y[n], \nu)}(w) = \Bigg(\sum_{\substack{(\a, \b) \\ |\a| + |\b| = |\nu|, }} \chi^{(\varnothing, \nu)}_{(\a,\b)} \; \chi^{(\y[n-m], \varnothing)}_{(\d, \g)} \; \binom{ \bX}{\a } \binom{ \bY} {\b } \Bigg)(w)$$
with $(\d, \g)$ the double partition of $(n-m)$ such that $(\a \cup \d, \b \cup \g)$ is the signed cycle type of $w$. 

We write $\nu = \mu[m]$, where $\mu$ is the partition obtained from $\nu$ by discarding the largest part; thus, by Lemmas \ref{LemSignCharacters} and  \ref{LemPullbackCharacters},
\begin{align*}
 \chi^{(\varnothing, \mu[m])}_{(\a, \b)} & = (-1)^{\ell(\b)} P^{(\mu,0)}(\a, \b) \\ 
&= (-1)^{\ell(\b)} \sum_{\substack{ \text{Partitions } \rho, \sigma \\ |\rho| + |\s| = |\mu| }} \frac{(-1)^{\ell(\s)} \; \chi^{\mu}(\rho \cup \s) }{z_{\s}} \;   \;  \binom{\bX + \bY}{\rho}(\a, \b) \\ 
&= (-1)^{\ell(\b)} \sum_{\substack{ \text{Partitions } \rho, \sigma \\ |\rho| + |\s| = |\mu| }} \frac{(-1)^{\ell(\s)} \; \chi^{\mu}(\rho \cup \s) }{z_{\s}} \;   \;  \prod_r \binom{ n_r(\a) + n_r(\b)}{n_r(\rho)}
\end{align*}

\noindent Moreover, since for each $r$ we have $$n_r(\d) = X_r(w) - n_r(\a) \qquad \text{ and } \qquad n_r(\g) = Y_r(w) - n_r(\b),$$ we can use Lemma \ref{LemPullbackCharacters}
to compute:
\begin{align*}
 \chi^{(\y[n-m], \varnothing)}_{(\d, \g)}  & =  P^{(\y,0)}(\d, \g) \\
& = \sum_{\substack{ \text{Partitions } \xi, \eta \\ |\xi| + |\eta| = |\y| }} \frac{(-1)^{\ell(\eta)} \; \chi^{\y}(\xi \cup \eta) }{z_{\eta}} \;   \;  \prod_r \binom{ X_r + Y_r}{n_r(\xi)}(\d, \g)\\
& = \sum_{\substack{ \text{Partitions } \xi, \eta \\ |\xi| + |\eta| = |\y| }} \frac{(-1)^{\ell(\eta)} \; \chi^{\y}(\xi \cup \eta) }{z_{\eta}} \;   \;  \prod_r \binom{ X_r - n_r(\a) + Y_r - n_r(\b)}{n_r(\xi)}(w)
\end{align*}

Putting these together, 
\begin{align*}
\chi^{(\y[n-m], \nu)}(w) =& \Bigg(\sum_{\substack{(\a, \b) \\ |\a| + |\b| = |\nu|, }} \chi^{(\varnothing, \nu)}_{(\a,\b)} \; \chi^{(\y[n-m], \varnothing)}_{(\d, \g)} \; \binom{ \bX}{\a } \binom{ \bY} {\b } \Bigg)(w)\\
= & \Bigg(\sum_{\substack{(\a, \b) \\ |\a| + |\b| = |\nu|, }} \binom{ \bX}{\a } \binom{ \bY} {\b } \Bigg) (-1)^{\ell(\b)} \Bigg( \sum_{\substack{ \text{Partitions } \rho, \sigma \\ |\rho| + |\s| = |\mu| }} \frac{(-1)^{\ell(\s)} \; \chi^{\mu}(\rho \cup \s) }{z_{\s}} \;   \;  \prod_r \binom{ n_r(\a) + n_r(\b)}{n_r(\rho)}\Bigg) \\ 
& \; \Bigg( \sum_{\substack{ \text{Partitions } \xi, \eta \\ |\xi| + |\eta| = |\y| }} \frac{(-1)^{\ell(\eta)} \; \chi^{\y}(\xi \cup \eta) }{z_{\eta}} \;   \;  \prod_r \binom{ X_r - n_r(\a) + Y_r - n_r(\b)}{n_r(\xi)} \Bigg)(w)
\end{align*}

which gives the desired formula. 

Note that the degree of $P^{(\y, \nu)}$ 
\begin{align*}\text{deg}(P^{(\y, \nu)}) &\leq \bigg(|\a|+|\b|+ \max_{\substack{ \text{Partitions } \xi, \eta \\ |\xi| + |\eta| = |\y| }}|\xi| \bigg)\\
&= (|\nu|+|\y|)\\ 
&=d \end{align*}

\noindent so deg$(P^{(\y, \nu)})$ is at most the size of the double partition $(\y, \nu)$, as claimed. 
\end{proof}

\subsection{Finite generation and character polynomials} \label{SectionFinGenCharPoly}

We can now use Theorem \ref{ThmWnIndependentCharacters} to prove the existence of character polynomials for finitely generated $\FI_{BC}$--modules in Theorem \ref{WnPersinomial}. As a consequence of Theorems \ref{SnPersinomial} and \ref{WnPersinomial}, we can determine a number of constraints on the structure of finitely generated $\FI_{\W}$--modules.

\begin{thm} \label{WnPersinomial} {\bf (Characters of finitely generated $\FI_{\W}$--modules are eventually polynomial).}
Let $k$ be a field of characteristic zero. Suppose that $V$ is a finitely generated $\FI_{BC}$--module with weight $\leq d$ and stability degree $\leq s$, or,  alternatively, suppose that $V$ is a finitely generated $\FI_D$--module with weight $\leq d$ such that $\Ind_{D}^{BC}\;V$ has stability degree $\leq s$. In either case, there is a unique polynomial $$F_V \in k[X_1, Y_1, X_2, Y_2, \ldots],$$ independent of $n$, such that the character of $\W_n$ on $V_n$ is given by $F_V$ for all $n \geq s +d$. The polynomial $F_V$ has degree $\leq d$, with deg$(X_i)=$deg$(Y_i)=i$. \end{thm} 

We remark that, by Theorem \ref{WnDiagramSizesReview}, $d$ is at most the degree of generation of $V$. 

\begin{proof}[Proof of Theorem \ref{WnPersinomial}] Assume first that $V$ is a finitely generated $\FI_{BC}$--module.

Recall that the class functions $X_i$, $Y_i$ are algebraically independent as functions on the disjoint union of all hyperoctahedral groups $\coprod_{n \geq 0} B_n$. The uniqueness of a character polynomial $F_V$ therefore follows from the general fact that any two (multivariate) polynomials that agree on all but finitely many integer points are necessarily equal. We turn to proving existence of the character polynomial $F_V$. 

  By Theorem \ref{StabilityDegreeImpliesRepStabilityReview}, for $n \geq s+d$, $V_n$ has a decomposition $$V_n = \bigoplus_{\y}c_{\y}V(\y)_n$$ where by assumption $c_{\y}$ is only nonzero for $|\y| \leq d$. Thus for $n \geq s+d$ the characters $V_n$ are given by a character polynomial of degree $\leq d$ by Theorem \ref{ThmWnIndependentCharacters}. 
 
 We will now use this result to prove the theorem for type D. That $V$ is an $\FI_D$--module of weight $\leq d$ means by definition that $\Ind_D^{BC}V$ is an $\FI_{BC}$--module of weight $\leq d$, and $\Ind_D^{BC}V$ moreover has stability degree $\leq s$ by assumption. Hence the $B_n$--representations $(\Ind_D^{BC}V)_n$ are given by a unique character polynomial $F_V$ for all $n \geq s +d$. Moreover, if $V$ is generated in degree $\leq m$, then $$V_n \cong (\Res_D^{BC} \Ind_D^{BC}V)_n \qquad  \text{for all $n \geq m+1$}$$ by Theorem \ref{VisResIndVReview}, and so the character of $V_n$ is given by the restriction of $F_V$ to $D_n$ in this range. The theorem follows.
\end{proof}

\begin{cor} {\bf (Polynomial growth of dimension for finitely generated $\FI_{\W}$--modules).} \label{PolyGrowth}
 Given a finitely generated $\FI_{\W}$--module $V$ over a field of characteristic zero with associated character polynomial $F_V$, the dimension dim$(V_n)$ of $V_n$ is given by $F_V(n,0,0,0, \ldots)$ in the stable range. In particular, if $V$ is finitely generated in degree $\leq d$, then dim$(V_n)$ is eventually a polynomial in $n$ of degree at most $d$.
\end{cor}

\begin{cor} {\bf (Characters only depend on short cycles).} Suppose that $k$ is a field of characteristic zero, and let $V$ be a finitely generated $\FI_{\W}$--module. Let $\chi_n$ denote the character of the $B_n$--representation $V_n$. Then there exists some positive integer $d \leq$ weight$(V)$, independent of $n$,  such that for every $w \in \W_n$, the value $\chi_n(w)$ depends only on cycles in $w$ of length at most $d$. 
\end{cor}

\begin{rem}{\bf (Character polynomials of co--$\FIW$--modules).}
 Suppose that $V$ is a co--$\FIW$--module over a field of characteristic 0. We define its dual $V^*$ to be the $\FIW$--module with $(V^*)_n = (V_n)^*$. Suppose $V^*$ is a finitely generated $\FIW$--module of weight $\leq d$ and stability degree $\leq s$, and that $F_V$ is the associated character polynomial. Since $(V_n)^* \cong (V_n)$ (see Geck--Pfeiffer \cite[Corollary 3.2.14]{GeckPfeiffer}), the characters of $\chi_{V_n} = F_V$ in the range $n\geq s+d$. 
\end{rem}

\subsection{Polynomial dimension over positive characteristic} \label{SectionPolyDimCharP}

Church--Ellenberg--Farb--Nagpal proved that the dimensions of finitely generated $\FI_A$--modules over a field $k$ are eventually polynomial even when $k$ has positive characteristic \cite[Theorem 1.2]{CEFN}. We use their result to prove the same for all $\FIW$--modules.

\begin{thm}\label{PolyDimCharP} {\bf (Polynomial growth of dimension over arbitrary fields).} Let $k$ be any field, and let $V$ be a finitely generated $\FIW$--module over $k$. Then there exists an integer-valued polynomial $P(T) \in \Q[T]$ such that $$ \dim_k (V_n) = P(n) \qquad \text{for all $n$ sufficiently large.}$$
\end{thm}
We note that, in contrast to the result over characteristic zero, Theorem \ref{PolyDimCharP} does not come with bounds on the degree of $P(T)$ or the range of $n$-values for which the equality holds. 

\begin{proof}[Proof of Theorem \ref{PolyDimCharP}] When $V$ is a finitely generated $\FI_A$--module, the result follows from \cite[Theorem 1.2]{CEFN}. If $V$ is a finitely generated $\FI_{BC}$ or $\FI_D$--module, then by Proposition \ref{RestrictionPreservesFinGenReview} its restriction to $\FI_A$ is finitely generated, and the result again follows from \cite[Theorem 1.2]{CEFN}. 
\end{proof}

\subsection{The character polynomials of $\FIW\sh$--modules} \label{CharacterPolynomialsFISharp}

In this section we compute the character polynomials of the $\FI_{BC}$--modules $M_{BC}(U)$, Proposition \ref{InducedCharacter}. We conclude that the character polynomial of an $\FI_{BC}\sh$--module $V$ must equal $\chi_{V_n}$ for \emph{all} values of $n$. The formula given in Proposition \ref{InducedCharacter} is moreover useful for computing character polynomials of $\FI_{BC}\sh$--modules, such as in our applications in Sections \ref{SectionPureStringMotion} and \ref{SectionHyperplaneComplements}. We end this section with Proposition \ref{CharacterM(m)}, the character polynomials of the $\FIW$--modules $M_{\W}(\bm)$ for each family $\W_n$.

\begin{prop} \label{InducedCharacter} {\bf(The Character of $M_{BC}(U)_n$).}
 Let $k$ be a field of characteristic zero. Let $U$ be a representation of $B_m$ with character $\chi^{U}$. Then the character $\chi^{M_{BC}(U)_n}$ is, for each $n$, given by the character polynomial $P^{U}$: 
\begin{align*} P^{U}(w) & = \Bigg(\sum_{\substack{(\a, \b) \\ |\a| + |\b| = m}}  \chi^{U}_{(\a, \b)} \binom{ \bX}{\a } \binom{ \bY} {\b }\Bigg)(w) \\ 
 & : = \sum_{\substack{(\a, \b) \\ |\a| + |\b| = m}}  \chi^{U}_{(\a, \b)} \binom{ X_1(w)}{n_1(\a) } \binom{ X_2(w)} {n_2(\a) } \cdots  \binom{ X_m(w)}{n_m(\a) } \binom{ Y_1(w)}{n_1(\b) } \binom{ Y_2(w)}{n_2(\b) } \cdots \binom{ Y_m(w)} {n_m(\b) }  
\end{align*}
\end{prop}

\begin{proof}[Proof of Proposition \ref{InducedCharacter}]
Since $\displaystyle M_{BC}(U)_n = \Ind^{B_n}_{B_m \times B_{n-m}} U \boxtimes k,$ with $k$ the trivial $B_{n-m}$--representation, the result follows from Lemma \ref{LemInducedRepFormula}. 
\end{proof}

\begin{cor} \label{CharPolyFISharp}
 Let $V$ be an $\FI_{BC} \sh$--module $V$ over a field of characteristic zero. Then if $V$ is finitely generated in degree $\leq d$, the characters of $V_n$ are equal to a unique character polynomial $F_V \in k[X_1, Y_1, X_2, Y_2, \ldots]$ of degree at most $d$, with equality for every value of $n \geq 0$. The dimensions of $V_n$ are given by a polynomial of degree at most $d$ $$ \dim_k(V_n) = F_V(n, 0, 0, \ldots ) \qquad \text{for every value of $n$.}$$
\end{cor}

We can find explicit formulas for the $\FIW$--modules $M_{\W}(\bm)$. 

\begin{prop} \label{CharacterM(m)} Let $k$ be $\Z$ or a field of characteristic zero. When $\W_n$ is $S_n$ or $B_n$, the character polynomials of $M_{\W}(\bm)$ are:  
$$\chi^{M_A(\bm)} = m! {X_1 \choose m} \qquad \qquad \chi^{M_{BC}(\bm)} = 2^m m! {X_1 \choose m}.$$
When $\W_n$ is $D_n$, $M_{D}(\bm)$ is also given by a character polynomial for $n>m$: 
$$\chi^{M_{D}(\bm)} = 2^m m! {X_1 \choose m} \qquad \text{when $n>m$}.$$ When $n=m$, the character of $M_{D}(\bm)_m$ takes the value $2^{m-1} m!$ on the identity and vanishes otherwise. 
\end{prop}

\begin{proof}[Proof of Proposition \ref{CharacterM(m)}] 
 Take as basis for $M_{\W}(\bm)_n$ the set $S = \{ e_f \; | \; f \in \Hom_{\FIW}(\bm, \bn) \}.$
 
 An element $w \in \W_n$ will permute these basis elements; the trace of $w$ is the size of its fixed set in $S$. A basis element $e_f$ is fixed by $w$ only if $w$ fixes its image $f(\bm) \subseteq \bn$ pointwise; conversely for every choice of $m$ (positive) $1$-cycles $$( a_1)(-a_1), (a_2)(-a_2), \ldots, (a_m)(-a_m)$$ in $w$, $w$ will fix all basis elements $e_f$ for which the image of $f$ is $$f(\bm) = \{ \pm a_1, \ldots, \pm a_m \} \subseteq \bn.$$
 
 When $\W_n$ is $S_n$, there are $m!$ such maps. When $\W_n$ is $B_n$, there are $2^m m!$ such maps. When $\W_n$ is $D_n$, there are  $2^m m!$ such maps whenever $n>m$; when $n=m$ there are only $2^{m-1} m!$, since in this case each endomorphism $f$ must reverse an even number of signs. The formulas follow. 
\end{proof}

\begin{ex} Using the formulas for the characters of $M_A(\bm)$ and $M_{BC}(\bm)$ in Proposition \ref{CharacterM(m)} and a standard binomial--multinomial identity, one can decompose the characters of $M_{A}(\bm) \otimes M_A(\bp)$ and $M_{BC}(\bm) \otimes M_{BC}(\bp)$. By Corollary \ref{FISharpDeterminedByCharacters}, these characters completely determine the $\FIW$--structure: 
$$M_{A}(\bm) \otimes M_A(\bp) = \bigoplus_{d=0}^m \frac{m! \; p!}{( m+p-d)!}{m+p-d \choose d,m-d,p-d} M_A({\bf m+p-d})$$
$$M_{BC}(\bm) \otimes M_{BC}(\bp) = \bigoplus_{d=0}^m  \frac{ 2^d \; m! \; p!}{( m+p-d)!}{m+p-d \choose d,m-d,p-d} M_{BC}({\bf m+p-d})$$
\end{ex}

\section{Some applications} \label{SectionApplications}

$\FIW$--modules arise naturally in numerous areas of mathematics. In this section we give some applications of the theory developed in this paper to the cohomology of the group of pure string motions  and the cohomology of the complements of the Weyl groups' reflecting hyperplanes. 

\subsection{The cohomology of the group of pure string motions} \label{SectionPureStringMotion}

In \cite{PureStringMotion}, we proved that the cohomology of the pure string motion groups $\PS_n$ is uniformly representation stable with respect to a natural action of the hyperoctahedral group. 

The group $\S_n$ of string motions is a generalization of the braid group. It is defined as the group of \emph{motions} of $n$ smoothly embedded, oriented, unlinked, unknotted circles in $\R^3$; see for example Brownstein--Lee \cite{BrownsteinLee} for a complete definition. 

Hatcher--Wahl \cite{HatcherWahl} proved that the string motion groups are (integrally) homologically stable, as a particular instance of their homologically stability results for mapping class groups of $3$--manifolds. Brendle--Hatcher \cite{BrendleHatcher} realized the string motion group as the fundamental groups of certain configuration spaces of rigid circles. These configuration spaces are not $K(\pi, 1)$'s, and Kupers proved that the spaces are themselves homological stable \cite{KupersEuclideanCircles}.

The work of Dahm (see Dahm \cite{Dahm} or Goldsmith \cite{Goldsmith}) identifies $\S_n$ with the \emph{symmetric automorphism group} of the free group $F_n$ on $n$ letters $x_1, \ldots, x_n$, the subgroup of automorphisms generated by the following elements: 
 \begin{align*}
 \a_{i,j} = \left\{ \begin{array}{ll}
        x_i \mapsto x_jx_ix_j^{-1} & \\
        x_{\ell} \mapsto x_{\ell} \quad \;  (\ell \neq i) &   \end{array} \right. 
 \t_i = \left\{ \begin{array}{ll}
        x_i \mapsto x_{i+1} & \\
        x_{i+1} \mapsto x_i & \\ 
        x_{\ell}  \mapsto x_{\ell}  \quad ({\ell}  \neq i,\, i+1) & \end{array} \right. 
 \p_i = \left\{ \begin{array}{ll}
        x_i \mapsto x_i^{-1}& \\
        x_{\ell}  \mapsto x_{\ell}  \quad ({\ell}  \neq i) & \end{array} \right.
\end{align*}
The subgroup $\PS_n = \langle \a_{i,j} \rangle \subseteq \S_n$ is the group of \emph{pure} symmetric automorphisms (or \emph{pure} string motions), the analogue of the pure braid group. 

The central theorem of \cite{PureStringMotion}:

\begin{thm} \cite[Theorem 6.1]{PureStringMotion}
For each fixed $m \geq 0$, the sequence of $B_n$--representations $\{H^m(\PS_n ; \Q)\}_{n}$ is uniformly representation stable with respect to the maps $$\phi_n \, : \, H^m(\PS_n ; \Q) \to H^m(\PS_{n+1} ; \Q)$$ induced by the `forgetful' map $\PS_{n+1} \to \PS_n$. The sequence stabilizes once $n \geq 4m$.
\end{thm}

The theory of $\FI_{BC}$--modules developed here allows for a significantly simplified proof of this result, and new perspective on the structure of these cohomology groups.

The integral homology group $H^1(\PS_n; \Z) = \PS_n / [ \PS_n, \PS_n]$ is the free abelian group $\Z [ \; \a_{i,j} \; \vert \; i \neq j \;]$, and the cohomology ring is generated by the dual elements $\a^*_{i,j}$. A presentation for the integral cohomology was conjectured by Brownstein and Lee \cite[Conjecture 4.6]{BrownsteinLee} and proven by Jensen, McCammond, and Meier \cite[Theorem 6.7]{JMM} (see also Griffin \cite[Section 4]{GriffinThesis}). Jensen--McCammond--Meier study the action of $\PS_n / \Inn(F_n)$ on the MacCullough--Miller complex \cite{MacculloughMiller} to obtain Theorem \ref{thmJMM}. 

\begin{thm}  \cite[Theorem 6.7]{JMM}. \label{thmJMM}
 The cohomology ring $H^*(\PS_n; \Z)$ is the exterior algebra generated by the degree-one classes $\a_{i,j}^*$, with $i,j \in [n]$, $i \neq j$, modulo the relations 
\begin{align*}
      (1) \;  \a^*_{i,j} \wedge \a^*_{j,i} = 0   \qquad \qquad  (2) \; \a^*_{\ell,j} \wedge  \a^*_{j,i} -  \a^*_{\ell,j} \wedge  \a^*_{\ell,i} + \a^*_{i,j} \wedge  \a^*_{\ell,i} = 0 
\end{align*}
\end{thm}

In \cite{PureStringMotion}, to prove that the sequence $H^m(\PS_n; \Q)$ is uniformly representation stable, we use a combinatorial description of the cohomology groups given by Jensen--McCammond--Meier \cite{JMM} and an orbit--stabilizer argument to decompose each group into a sum of induced representation of a particular form. We then use a result of Church--Farb \cite[Theorem 4.6]{RepStability} (inspired by the work of Hemmer \cite[Theorem 2.4]{Hemmer}), to deduce from the combinatorics of the branching rules that these induced representations are uniformly representation stable. 

Here, we can recover uniform representation stability for $H^m(\PS_{\bullet};\Q)$ as a $B_n$--representation almost immediately by demonstrating that it is finitely generated as an $\FI_{BC}\sh$--module, as follows.

\begin{thm} \label{PureStingMotionisFI-Sharp} 
Let $k$ be $\Z$ or $\Q$. The cohomology rings $ H^*(\PS_{\bullet},  k)$ form an $\FI_{BC} \sh$--module, and a graded $\FI_{BC}$--algebra of finite type, with  $ H^m(\PS_{\bullet},  k)$ finitely generated in degree $\leq 2m$. In particular the $\FI_{BC}$--algebra $ H^*(\PS_{\bullet},  \Q)$ has slope $\leq 2$.
\end{thm}

\begin{proof}[Proof of Theorem \ref{PureStingMotionisFI-Sharp}]
 The map induced by an $\FI_{BC}\sh$ morphism $f: \bmo \to \bno$ on $H^{1}(\PS_{\bullet},  \Z)$ is:
\begin{align*}
f_*  : \;  H^1(\PS_m; k) \quad \longrightarrow & \quad H^1(\PS_n; k)  \\ 
  \a_{i,j}^* \longmapsto & \left\{ \begin{array}{ll} \; \; \, \a_{|f(i)|, |f(j)|}^*,\; \text{ if $f(i) \neq 0, f(j)>0$} & \\ 
					 - \a_{|f(i)|, |f(j)|}^*,\; \text{if $f(i) \neq 0, f(j)<0$} & \\
                                                                    \; \; \;  0, \qquad \qquad \; \;  \text{if $f(i)=0$ or $f(j)=0$} & \end{array} \right.
\end{align*}

\noindent It is straightforward to verify that $f_*$ extends to an algebra map on $ H^*(\PS_{\bullet},  k)$, and that this action is functorial.  

The $\FI_{BC}$--module $H^1(\PS_{\bullet};k)$ is generated in degree $2$ by $\a_{1,2}$, and  $V^*_{\bullet} = H^*(\PS_{\bullet}; k)$ is generated as an $\FI_{BC}$--algebra by $H^1(\PS_{\bullet}; k)$. We conclude from Proposition \ref{SlopeFinGenFromGeneratorsReview} that $V_{\bullet}^m$ is finitely generated in degree $\leq 2m$, and that $H^*(\PS_{\bullet}; \Q)$ is a graded $\FI_{BC}$--algebra of slope $\leq 2$. \end{proof}

Propositions \ref{RestrictionPreservesFinGenReview}(\ref{WntoSnReview}) and (\ref{WntoW'nReview}) imply that $V^*_{\bullet}$ restricts to a graded $\FI_A$ and $\FI_D$--algebra of finite type, with $V^m_{\bullet}$ generated in degree $\leq 2m$. 


Since $V^m_{\bullet} = H^m(\PS_{\bullet}; k)$ is a $\FI_{BC}\sh$--module generated in degree $\leq 2m$, by Corollary \ref{SharpRepStableDeg2d} the sequence is uniformly representation stable in degree $4m$, as $B_n$--representations or as $S_n$--representations. 

\begin{cor} \label{PSnRepStability} For each $m$, the sequence $ \{ H^m(\PS_{n};\Q)\}_n$ of representations of $B_n$ (or $S_n$) is uniformly representation stable, stabilizing once $n \geq 4m$. 
\end{cor}

\begin{rem}
We only defined \emph{representation stability} for $\FIW$--modules over fields of characteristic zero (Definition \ref{DefnRepStability}). However, the presentation for the groups $H^m(\PS_{n}; \Z)$ given by Jensen-McCammond--Meier shows that these groups are free abelian (\cite[Theorem 6.7]{JMM}, see Theorem \ref{thmJMM}), and we can identify $H^m(\PS_{n}; \Z)$ with the integer span of the basis $\a_{i_1, j_1}^* \wedge \cdots \wedge \a_{i_m, j_m}^*$ for $H^m(\PS_{n}; \Q)$. Hence we get a version of representation stability for the integral groups $H^m(\PS_{n}; \Z)$ by redefining the subrepresentation $V(\y)_n$ as the integral Specht module associated to $\y[n]$. 
\end{rem}

Theorem \ref{WnPersinomial} implies that the characters of the sequence $\{ H^m(\PS_{n};\Q) \}_n$ are given by a character polynomial of degree $\leq 2m$. As in the above remark, since the integral cohomology is free abelian, these same character polynomials give characters for the integral cohomology. 


\begin{cor} \label{PSnCharPoly} Let $k$ be $\Z$ or $\Q$. Fix an integer $m \geq 0$. The characters of the sequence of $B_n$--representations $\{ H^m(\PS_{n}; k) \}_n$ are given, for all values of $n$, by a unique character polynomial of degree $\leq 2m$.  
\end{cor}

This concludes a simpler proof of \cite[Theorems 6.1 and 6.4]{PureStringMotion}. We have moreover extended the results of \cite{PureStringMotion} to integer coefficients, and obtained polynomiality results on the characters.  

For small values of $m$, we can compute the $\FI_{BC} \sh$ and $\FI_A \sh$--module structures  and character polynomials of $H^m(\PS_{\bullet}; \Z)$ by computing traces on an explicit basis for $H^m(\PS_n; \Z)$, $n=1, \ldots, 2m$, and using Proposition \ref{InducedCharacter}. The result is, as an $\FI_{BC} \sh$--module,
\begin{align*} 
H^1(\PS_{\bullet}; \Z)  = M_{BC} \big(\; \Y{1} \;, \Y{1}\; \big) \qquad \qquad 
\chi_{H^1(\PS_{\bullet}; \Z)}   = 2 { X_1 \choose 2 } - 2 { Y_1 \choose 2 } 
\end{align*}
In degree $2$:
\begin{align*} 
H^2(\PS_{\bullet}; \Z)  = & \; M_{BC} \big(\; \Y{1,1,1} \; , \varnothing \;\big) \oplus M_{BC} \big(\; \Y{2,1} \; , \varnothing \; \big) \oplus M_{BC} \big(\; \Y{1} \; , \Y{1,1} \;\big) \oplus M_{BC} \big(\; \Y{1} \; , \Y{2} \; \big) 
\\& \oplus M_{BC} \big(\; \Y{1,1} \; , \Y{2} \; \big)  \oplus M_{BC} \big(\; \Y{2} \;, \Y{1,1} \; \big)  \\  \empty \\ 
\chi_{H^2(\PS_\bullet; \Z)}   = & \; 12 { X_1 \choose 4} + 12 {Y_1 \choose 4 } + 9 { X_1 \choose 3 } + 9 {Y_1 \choose 3} -4 {X_2 \choose 2} + 4 {Y_2 \choose 2}  \\& -4 {X_1 \choose 2}{Y_1 \choose 2} -X_1X_2 - X_1Y_2 - X_2Y_1 - Y_1Y_2 - {X_1 \choose 2}Y_1 - X_1{Y_1 \choose 2}  \end{align*}

By restricting to the action of the symmetric groups we find, as an $\FI_A \sh$--module,
\begin{align*} 
H^1(\PS_{\bullet}; \Z)  = M_{A} \big(\; \Y{2} \; \big) \oplus M_{A} \big( \;  \Y{1,1}\; \big) \qquad \qquad 
\chi_{H^1(\PS_{\bullet}; \Z)}   = 2 { X_1 \choose 2 }  
\end{align*}
\begin{align*} 
H^2(\PS_{\bullet}; \Z)  = & \; M_{A} \big( \; \Y{1,1,1} \; \big)^{\bigoplus 2} \oplus M_{A} \big( \; \Y{2,1} \; \big)^{\bigoplus 3} \oplus M_{A} \big( \; \Y{3} \; \big) \oplus M_{A} \big( \; \Y{2,1,1} \; \big)^{\bigoplus 2} \oplus M_{A} \big( \; \Y{3,1} \; \big)^{\bigoplus 2} \\  \empty \\ 
\chi_{H^2(\PS_\bullet; \Z)}   = & \; 12 { X_1 \choose 4} + 9 { X_1 \choose 3 } - X_1X_2 -4 {X_2 \choose 2} 
\end{align*}

\begin{prob} For each $m$, compute the $B_n$ character polynomial of $H^m(\PS_{\bullet}, \Z)$, and compute its decomposition as an $\FI_{BC}\sh$--module into a sum of  induced representations $M_{BC}(U)$.  \end{prob}

\subsubsection{Generalizations}

There are several families of groups that naturally generalize the (pure) braid groups and (pure) symmetric automorphism groups, which we outline below. With each family, there are open questions concerning whether the cohomology rings admit the structure of a finite type $\FIW$ or $\FIW\sh$--algebra, and how this structure reflects the structure of the groups. \\

{\noindent \bf Partially symmetric automorphisms. \quad} The group $\S_n^k$ of \emph{partially symmetric automorphisms} of the free group $F_n=\langle x_1, \ldots, x_n \rangle$  are those automorphisms that send each of the first $k$ generators $x_1, \ldots, x_k$ to a conjugate of one of the elements $x_1, x_1^{-1},$ $\ldots, x_k, x_k^{-1}$. We impose no restrictions on the images of $x_{k+1}, \ldots, x_n$. The \emph{pure partially symmetric automorphism group} $\PS_n^k$ is the subgroup of $\S_n^k$ of automorphisms that send each generator $x_j$ with  $1 \leq j \leq k$ to a conjugate of itself. We note that $$\S_n^n = \S_n, \qquad  \PS_n^n = \PS_n, \quad \text{and} \quad \PS_n^0 = \S_n^0 = \Aut(F_n);$$ these groups interpolate between the (pure) symmetric automorphism group and the full automorphism group of $F_n$. 

The groups $\PS_n^k$ were studied by Jensen--Wahl \cite{JensenWahlAutomorphisms} for their relationships to mapping class groups. Jensen--Wahl have computed a presentation and established certain homological properties of the groups. Bux--Charney--Vogtmann \cite{BuxCharneyVogtmann} determined that the image of the group $\PS_n^k$ in Out$(F_n)$ has virtual cohomological dimension $2n - k - 2$ when $k \neq 0$ . They exhibit a proper action of these outer automorphism groups on a $(2n - k - 2)$--dimensional deformation retract of a certain contractible subcomplex of the spine of Culler--Vogtmann's Outer space; see Charney--Vogtmann \cite{CharneyVogtmannFiniteness} for details.

Zaremsky \cite{ZaremskyPartial} proved that both families $\PS_n^k$ and $\S_n^k$ are, for fixed $k$, rationally homologically stable in $n$. He proved moreover that for fixed $n$, the groups $\S_{n+k}^k$ are rationally homologically stable in $k$. Zaremsky obtains these results by studying the groups' actions on subcomplexes of the spine of Auter space. He uses methods from discrete Morse theory to prove that the filtered pieces of certain subcomplexes are highly connected, extending techniques of McEwen--Zaremsky \cite{McEwenZaremsky}.  

Given these results, it would be interesting to determine whether there is a $\FI_{BC}$ or $\FI_{BC}\sh$--module structure on the rational cohomology groups of $\PS_{n+k}^k$ as a sequence in $k$, and, if so, to determine the associated stable decompositions and character polynomials. \\

{\noindent \bf Symmetric automorphisms of free products. \quad} Given a group $G$, let $\Gn$ denote its $n$-fold free product
$$\Gn := \underbrace{ G * G * \cdots * G }_{n \text{ copies}}. $$ 
 The automorphism group $\Aut(\Gn)$ contains a copy of the symmetric group $S_n$ which permutes the $n$ free factors. These permutations normalize the following subgroups of automorphisms; see for example Griffin \cite{Griffin, GriffinThesis} for details.
\begin{itemize}
 \item The \emph{Fouxe-Rabinovitch group} FR$(\Gn) \subseteq \Aut(\Gn)$ generated by \emph{partial conjugations} of $\Gn$. A partial conjugation is an  automorphism that conjugates the $i^{th}$ free factor $G$ by some $g$ in the $j^{th}$ factor $G$ with $i \neq j$. All factors other than the $i^{th}$ are fixed. 
 \item The inner automorphisms of each factor $\displaystyle \prod_n \Inn(G)$
 \item All automorphisms of each factor $\displaystyle \prod_n \Aut(G)$
 \item The \emph{Whitehead automorphism group} Wh$\displaystyle (\Gn):=  \text{FR}(\Gn) \rtimes \prod_n \Inn(G)$
 \item The \emph{pure automorphism group} P$\displaystyle \Aut(\Gn):=  \text{FR}(\Gn) \rtimes \prod_n \Aut(G)$
\end{itemize}
 The \emph{symmetric automorphism group} of $\Gn$ is the group $$\S\Aut(\Gn) := \big( \mathrm{P}\Aut(\Gn) \rtimes S_n \big).$$ We note that $$\PS_n \cong \text{FR}(\Z^{*n}) \cong \text{Wh} (\Z^{*n}) \qquad \text{and} \qquad \S_n \cong \S\Aut(\Z^{*n}). $$
 
 Griffin constructs a classifying space for FR$(\Gn)$, which he defines as a moduli space of \emph{cactus products}, and alternatively characterizes combinatorially in terms of \emph{diagonal complexes} comprised of \emph{forest posets}. Using this classifying space he computes the homology of the groups FR$(\Gn)$, P$\Aut(\Gn)$, and $\S\Aut(\Gn)$.

Collinet--Djament--Griffin \cite{CollinetDjamentGriffin} have proven that if $G$ does not contain $\Z$ as a free factor, the sequences $\Aut(\Gn)$ and $\S \Aut(\Gn)$ are (integrally) homologically stable, stabilizing in degree $i$ once $n \geq 2i+2$. Their work complements the results of Hatcher \cite{HatcherStabilityAut(Fn)} for $G\cong \Z$  and extends results of Hatcher--Wahl \cite{HatcherWahl} for several important classes of groups $G$ coming from low-dimensional topology. Collinet--Djament--Griffin prove their results using the theory of functor homology, and an analysis of the action of FR$(\Gn)$ on a variation of the MacCullough--Miller complex \cite{MacculloughMiller} due to Chen--Glover--Jensen \cite{ChenGloverJensen}. 

We would be interested to better understand the relationship between the work done on the groups FR$(\Gn)$,  Wh$\displaystyle (\Gn)$, and P$\Aut(\Gn)$ and the theory of $\FI_A$--modules.  \\ 

{\noindent \bf Virtual and flat braid groups. \quad}  The \emph{(pure) virtual braid group} and the \emph{(pure) flat braid group} are generalizations of the (pure) braid group that allow \emph{virtual} or \emph{flat} crossings of strands.  This additional structure was introduced by Kauffman \cite{KauffmanVirtual}, motivated by the study of knots in thickened higher-genus surfaces and the combinatorial theory of Gauss codes. Virtual and flat crossings are distinct from the under-- and over--crossings in familiar knot and braid diagrams, and each have their own admissible Redemeister moves. For details see for example Kauffman \cite{KauffmanVirtual, KauffmanSurvey}, Vershinin \cite{VershininVirtual}, Kauffman--Lambropoulou \cite{KauffmanLambropoulou}, Bardakov \cite{BardakovVirtual}, and Bartholdi--Enriquez--Etingof--Rains \cite{BaEnEtRa}.

In \cite{LeeRepStable}, Peter Lee analyzes the cohomology of the pure virtual braid groups and the pure flat braid groups as representations of the symmetric groups. He proves that, for both families, the rational cohomology groups are uniformly representation stable \cite[Corollaries 1 and 5]{LeeRepStable}, and his work suggests that these cohomology sequences are in fact $\FI_A\sh$--algebras. His results raise questions about the structure of these $\FI_A\sh$--algebras and their associated character polynomials, and the structure of the integral cohomology groups.

\subsection{The cohomology of hyperplane complements} \label{SectionHyperplaneComplements}

Let $\W_n$ be the Weyl group in type A$_{n-1}$, B$_n$/C$_n$, or D$_n$, and consider the canonical action of $\W_n$ on $\C^n$ by (signed) permutation matrices. Let $\cA(n)$ be the set of hyperplanes fixed by the (complexified) reflections of $\W_n$, and let $\cM_{\W}=\cM_{\W}(n)$ be their complement $$\cM_{\W}(n) := \C^n \backslash \bigcup_{H \in \cA(n)} H.$$ The group $\W_n$ permutes the set of hyperplanes, and acts on $\cM_{\W}$. For each family $\{\W_n\}$, the hyperplane complements can be described explicitly:
\begin{align*}
\cM_A(n) &= \left\{ (v_1, \ldots, v_n) \in \C^n \; \vert \; v_i \neq v_j \text{ for } i \neq j  \right\} \\
\cM_{D}(n) &= \left\{ (v_1, \ldots, v_n) \in \C^n \; \vert \; v_i \neq \pm v_j \text{ for } i \neq j   \right\}  \\
\cM_{BC}(n) &= \left\{ (v_1, \ldots, v_n) \in \C^n \; \vert \; v_i \neq \pm v_j \text{ for } i \neq j ;  v_i \neq 0 \text{ for all } i  \right\} 
\end{align*}
We note that $\cM_{BC}(n) \subseteq \cM_{D}(n) \subseteq \cM_{A}(n)$.

The hyperplane complement $\cM_A(n)$ is the ordered $n$--point configuration space of the plane $\C$; it is an Eilenberg--Mac Lane space with fundamental group the pure braid group on $n$ strands. Arnol'd computed its integral cohomology in 1969 \cite{Arnol'd}. Its quotient $\cM_A(n)/S_n$ is an Eilenberg--Mac Lane space with fundamental group the braid group on $n$ strands.  Brieskorn showed that $\cM_{BC}(n)$ and $\cM_{D}(n)$ and their quotients $\cM_{BC}(n)/B_n$ and $\cM_{D}(n)/S_n$ are also Eilenberg--Mac Lane spaces \cite[Proposition 2]{Brieskorn}; their fundamental groups are sometimes called \emph{generalized (pure) braid groups}.

Brieskorn \cite{Brieskorn} and Orlik--Solomon \cite{OrlikSolomon} studied the cohomology of the complement $\cM$ of a general arrangement of complex hyperplanes containing the origin. Define a set of hyperplanes $H_1, \ldots , H_p$ to be \emph{dependent} if $ \codim(H_1 \cap \cdots \cap H_p) < p.$

Let $E(\cA)$ to be the complex exterior algebra 
$ E(\cA) := \bigwedge \langle e_H \; \vert \; H \in \cA \rangle $ and let $I(\cA) \subseteq E(\cA)$ be the ideal 
$$ I(\cA) := \langle \quad \sum^p_{\ell=1 } (-1)^\ell \; e_{H_1} \cdots \widehat{e_{H_{\ell}}} \cdots e_{H_p} \quad \vert \quad H_1, \ldots, H_p \text{ dependent} \quad  \rangle $$ 

Orlik--Solomon proved that $H^*(\cM_{\W}, \C)$ is isomorphic to $A(\cA) := E(\cA) / I(\cA)$ as a graded algebra \cite[Theorem 5.2]{OrlikSolomon}. Their work implies that $H^*(\cM_{\W}, \C) \cong A(\cA)$ as a graded $\C[\W_n]$--module under the $\W_n$--action $ w \cdot e_H  = e_{w H}.$ \

The structure of $H^*(\cM_{A}(n), \C)$ as an $S_n$--representation is described by Lehrer--Solomon \cite{LehrerSolomon}, and the structure of the $B_n$--representations $H^*(\cM_{BC}(n), \C)$ is described by Douglass \cite{DouglassArrangement}. Lehrer--Solomon and Douglass give decompositions of the $\W_n$--representations $H^*(\cM_{\W}(n), \C)$ in type $A$ and $B/C$, respectively, as sums of certain explicitly described induced representations. Lehrer--Solomon conjectured that, as they prove in type A, the cohomology groups $H^m(\cM_{\W}, \C)$ decompose into a sum of induced one-dimensional representations of centralizers, summed over the set of $\W_n$ conjugacy classes \cite[Conjecture 1.6]{LehrerSolomon}. Recent progress 
has been made on this conjecture; see Douglass--Pfeiffer--R\"{o}hrle \cite{DouglassPfeifferRohrle}.

  Church and Farb prove that, for each degree $m$, the sequence $H^m(\cM_A(n), \Q)$ is a uniformly representation stable sequence of $S_n$--representations \cite[Theorem 4.1]{RepStability}. Church--Ellenberg--Farb further prove that $H^m(\cM_A, \Q)$ is a graded $\FI\sh$--algebra of finite type; this is a special case of their much more general results on the ordered configuration space of manifolds \cite[Theorem 4.7; see also Theorems 4.1 and  4.2]{CEF}. In  \cite[Theorem 4.6]{RepStability}, Church--Farb analyze the stability behaviour of the sequence $H^m(\cM_{BC}, \C)$ of $B_n$--representations. 

The following result recovers \cite[Theorem 4.1 and 4.6]{RepStability} in types A$_{n-1}$ and B$_n$/C$_n$. It extends the work of Church--Ellenberg--Farb on the cohomology of the ordered configuration space of $\C$. 

\begin{thm} \label{HyperplaneCohomologyIsFISharp} 
Let $\cM_{\W}$ be the complex hyperplane complement associated with the Weyl group $\W_n$ in type A$_{n-1}$, B$_n$/C$_n$, or D$_n$, as described above. In each degree $m$, the cohomology groups $H^m(\cM_{A}(\bullet), \C)$ form an $\FI_A\sh$--module finitely generated in degree $\leq 2m$, and both $H^m(\cM_{BC}(\bullet), \C)$ and $H^m(\cM_{D}(\bullet), \C)$ are $\FI_{BC}\sh$--modules finitely generated in degree $\leq 2m$.  For each $\W$, the cohomology $H^*(\cM_{\W} (\bullet), \C)$ of the hyperplane complements is a graded $\FIW$--module of slope $2$.  
\end{thm}
\begin{proof}[Proof of Theorem \ref{HyperplaneCohomologyIsFISharp}]
For each $\W$, the projection map $P$ has a section $S$:
\begin{align*}
  P: \cM_{\W}(n+1) & \longrightarrow \cM_{\W}(n) &   S: \cM_{\W}(n) & \longrightarrow \cM_{\W}(n+1) \\
 (v_1, \ldots, v_n, v_{n+1}) & \longmapsto (v_1, \ldots, v_n)  & (v_1, \ldots, v_n) & \longmapsto (v_1, \ldots, v_n, 1+ \sum_{i=1}^n |v_i| ) 
\end{align*}
and so $P$ induces an injective map on cohomology, as follows. We associate each hyperplane $H \subseteq \C^n$ to its orthogonal complement, the span of the vectors $$\pm(\be_i-\be_j), \quad \pm(\be_i+\be_j), \quad  \text{or} \quad \pm \be_i \qquad  \text{for $i,j=1, \ldots, n$.}$$ The inclusion of these normal vectors $\C^n \to \C^{n+1}$ gives an identification of the hyperplane $H \subseteq \C^n$ with a hyperplane $H \subseteq \C^{n+1}$, which define the induced map $P^*$.
\begin{align*}
 P^*: H^*(\cM_{\W}(n); \C) & \longrightarrow H^*(\cM_{\W}(n+1), \C) \\ 
  e_H  & \longmapsto e_H 
\end{align*}
 These inclusions are $\W_n$--equivariant maps, and give $H^*(\cM_{\W}(\bullet); \C)$ the structure of a graded $\FIW$--module. 

The $\FI_A$--module $H^1(\cM_{A}(\bullet); \C)$ is finitely generated in degree $\leq 2$ by element $e_{(\be_1-\be_2)^{\perp}}$, and the $\FI_{BC}$--module $H^1(\cM_{BC}(\bullet); \C)$ is finitely generated in degree $\leq 2$ by elements $e_{(\be_1-\be_2)^{\perp}}$, $e_{(\be_1+\be_2)^{\perp}}$, and $e_{(\be_1)^{\perp}}$. It follows from Proposition \ref{SlopeFinGenFromGeneratorsReview} that $H^m(\cM_{\W}(\bullet); \C)$ is finitely generated in degree $\leq 2m$ in types A and B/C. The bound on the slope of the $\FIW$--algebra $H^*(\cM_{\W}(\bullet); \C)$ follows from Theorem \ref{WnDiagramSizesReview}. 


The section $S$ induces a map $S^*: H^*(\cM_{\W}(n+1); \C)  \longrightarrow H^*(\cM_{\W}(n), \C);$ when $\W_n$ is $S_n$ or $B_n$ these sections give $H^*(\cM_{\W}(\bullet), \C)$ the structure of an $\FIW\sh$--module, just as in the proof of \cite[Theorem 4.6]{CEF}. We can describe this structure explicitly: an $\FI_{BC}\sh$--morphism $f: \bmo \to \bno$ acts on the generators $e_H$ as follows.
\begin{align*}
  e_{(\be_i-\be_j)^{\perp}} \longmapsto & \left\{ \begin{array}{ll} \; \; \, e_{(\be_{f(i)}-\be_{f(j)})^{\perp}},\; \text{ if $f(i), f(j) \neq 0$} & \\
                                                                    \; \; \;  0, \qquad \qquad \quad \; \; \text{if $f(i)=0$ or $f(j)=0$} & \end{array} \right. \end{align*}\begin{align*}
  e_{(\be_i+\be_j)^{\perp}} \longmapsto & \left\{ \begin{array}{ll} \; \; \, e_{(\be_{f(i)}+\be_{f(j)})^{\perp}},\; \text{ if $f(i), f(j) \neq 0$} & \\
                                                                    \; \; \;  0, \qquad \qquad \quad \; \, \text{if $f(i)=0$ or $f(j)=0$} & \end{array} \right. 
\end{align*}\begin{align*}
  e_{(\be_i)^{\perp}} \longmapsto & \left\{ \begin{array}{ll} \; \; \, e_{(\be_{f(i)})^{\perp}},\; \text{ if $f(i) \neq 0$} & \\
                                                                    \; \; \;  0, \qquad  \quad  \, \text{if $f(i)=0$} & \end{array} \right.
\end{align*}
Here, we use the convention that $\be_{-i} := -\be_i$. It is straightforward to verify that these maps are functorial. 

In type $A$, this action restricts to an $\FI_{A}\sh$--module structure on the ring $H^*(\cM_{A}(n), \C)$ generated by the elements $e_{(\be_i-\be_j)^{\perp}}$. For type D, observe that the inclusion of hyperplane complements 
            $ \cM_{BC}(n)  \hookrightarrow \cM_{D}(n) $
induces an inclusion of cohomology groups 
             $ H^*(\cM_{D}(n); \C)  \hookrightarrow H^*(\cM_{BC}(n), \C). $

The subspaces $ H^*(\cM_{D}(n); \C) \subseteq H^*(\cM_{BC}(n), \C)$ form the $B_n$--invariant subring generated by the elements $e_{(\be_i-\be_j)^{\perp}}$ and $e_{(\be_i+\be_j)^{\perp}}$, $i\neq j$. These inclusions realize $H^*(\cM_{D}(n); \C)$ as a sub--$\FI_{BC}\sh$--module of $H^*(\cM_{BC}(n); \C)$ generated as an $\FI_{BC}$--algebra by the $\FI_{BC}$--module  $H^1(\cM_{D}(n); \C)$. Since $H^1(\cM_{D}(n); \C)$ is finitely generated in degree $ \leq 2$, it follows again that $H^*(\cM_{D}(n); \C)$ is an $\FI_{BC}\sh$--algebra of slope $2$ with $H^m(\cM_{D}(n); \C)$ finitely generated in degree $\leq 2m$. \end{proof}

Theorem \ref{HyperplaneCohomologyIsFISharp} has the following consequences.

\begin{cor} \label{HypRepStable}
  In each degree $m$, the sequence of cohomology groups $\{H^m(\cM_{\W}(n), \C)\}_n$ of the associated hyperplane complement is uniformly representation stable in degree $\leq 4m$. 
\end{cor}

In types A and B/C, Corollary \ref{HypRepStable} recovers \cite[Theorem 4.1 and 4.6]{RepStability}.

\begin{cor}  \label{HyperplaneComplementCharPoly}
  In each degree $m$, the sequence of characters of the $\W_n$--representations $H^m(\cM_{\W}(n), \C)$ are given by a unique character polynomial of degree $\leq 2m$ for all $n$. 
\end{cor}

\begin{proof}[Proof of Corollary \ref{HyperplaneComplementCharPoly}]
 The statement follows for $S_n$ from \cite[Theorem 2.67]{CEF}, and in type $B_n$ from Proposition \ref{WnPersinomial}.  Since the $D_n$ characters are the restriction of the characters of $B_n$ on the $B_n$--subrepresentations $H^*(\cM_{D}(n); \C) \subseteq H^*(\cM_{BC}(n), \C)$, these $D_n$ characters are given by the character polynomial for $B_n$ on this sub--$\FI_{BC}\sh$--module of $H^*(\cM_{BC}(\bullet), \C)$.
\end{proof}

The character polynomials for $H^m(\cM_{A}(\bullet), \C)$ are computed in \cite{CEF} for some low values of $m$. The decompositions for $H^1(\cM_{D}(\bullet), \C)$ and $H^1(\cM_{BC}(\bullet), \C)$ are:
\begin{align*} 
H^1(\cM_{D}(\bullet), \C)  = 2 M_D \big( \{ \; \Y{2} \; , \varnothing \; \} \big) \qquad \qquad 
\chi_{H^1(\cM_{D}(\bullet), \C)}   = 2 { X_1 \choose 2 } + 2 { Y_1 \choose 2 } + 2 X_2
\end{align*}
\begin{align*} 
H^1(\cM_{BC}(\bullet), \C) & =  M_{BC} \big(\; \Y{1} \; , \varnothing \; \big) \oplus M_{BC} \big(\; \Y{2} \; , \varnothing \; \big) \oplus M_{BC} \big(\; \varnothing \; , \; \Y{2} \; \big) \\ 
\chi_{H^1(\cM_{BC}(\bullet), \C)}  & = 2 { X_1 \choose 2 } + 2 { Y_1 \choose 2 } + 2 X_2 + X_1 - Y_1 
\end{align*}

The decompositions for $H^2(\cM_{D}(\bullet), \C)$ and $H^2(\cM_{BC}(\bullet), \C)$ are:
\begin{align*} 
H^2(\cM_{D}(\bullet), \C) & =  M_{D} \big( \{ \;  \Y{2} \; , \varnothing \; \} \big) \oplus  M_{D} \big(\{ \; \Y{2,1} \;, \varnothing \; \} \big) \oplus  M_{D} \big( \{ \;  \Y{1} \;, \Y{2} \; \} \big) \oplus  M_{D} \big( \{ \;  \Y{1} \;, \Y{1,1} \; \} \big)\\ 
 & \oplus M_{D} \big( \{ \;  \Y{3,1} \;, \varnothing  \; \} \big)^{\oplus 2} \oplus M_{D} \big(\;  \Y{2} \;, +  \; \big) \oplus M_{D} \big(\;  \Y{2} \;, -  \; \big)
\end{align*}
\begin{align*} 
\chi_{H^2 (\cM_{D}(\bullet), \C)}   = \; &  {X_1 \choose 2} - X_1 X_2 +  {Y_1 \choose 2} +  X_2 - Y_2 + 8 {X_1 \choose 3}
+ 8 {Y_1 \choose 3} - X_3 - Y_3 +  12 { X_1 \choose 4 } \\ &  
+ 4 {X_1 \choose 2} {Y_1 \choose 2} + 12 {Y_1 \choose 4} + 4 X_2 {X_1 \choose 2} + 4 X_2 {Y_1 \choose 2} - 4 {Y_2 \choose 2} -2 Y_4 \end{align*}

\begin{align*} 
H^2(\cM_{BC}(\bullet), & \C)  =  M_{BC} \big(\; \Y{2} \; , \varnothing \; \big) \oplus M_{BC} \big(\;  \varnothing \;, \Y{2} \; \big)^{\oplus 2} \oplus  M_{BC} \big(\; \Y{2,1} \;, \varnothing \; \big)^{\oplus 2} \\ &  \oplus   M_{BC} \big(\;  \Y{1} \;, \Y{2} \; \big)^{\oplus 2}  
\oplus  M_{BC} \big(\;  \Y{1} \;, \Y{1,1} \; \big)   \oplus M_{BC} \big(\;  \Y{3} \;, \varnothing \; \big) 
\oplus M_{BC} \big(\;  \varnothing \;, \; \Y{3,1} \; \big) \\& \oplus M_{BC} \big(\;  \Y{3,1} \;, \varnothing \; \big)  \oplus M_{BC} \big(\;  \Y{2} \;, \Y{2} \; \big)
\end{align*}
\begin{align*} 
\chi & _{ H^2 (\cM_{BC}(\bullet), \C)}   = 3 {X_1 \choose 2} + 3 {Y_1 \choose 2} - X_1Y_1 + 3 X_2 - Y_2 + 14 {X_1 \choose 3} + 2 {X_1 \choose 2} Y_1  \\ &
 + 2 X_1 {Y_1 \choose 2}    + 14 {Y_1 \choose 3}   + 2 X_2 X_1 + 2 X_2 Y_1 - X_3 - Y_3 + 12 { X_1 \choose 4 } 
 + 4 {X_1 \choose 2} {Y_1 \choose 2} \\ & + 12 {Y_1 \choose 4}  + 4 X_2 {X_1 \choose 2} + 4 X_2 {Y_1 \choose 2} - 4 {Y_2 \choose 2} -2 Y_4 
\end{align*}

\begin{prob} For each $m$, compute the character polynomial of the $\FI_A\sh$--module $H^m(\cM_{A}(\bullet), \C)$, and compute its decomposition into induced representations $M_A(U)$. Compute the character polynomials of the $\FI_{BC}\sh$--modules $H^m(\cM_{BC}(\bullet), \C)$ and $H^m(\cM_{D}(\bullet), \C)$ for each $m$, and compute the decomposition into induced representations $M_{BC}(U)$. \end{prob}



{\footnotesize 


\bibliographystyle{alpha}
\bibliography{MasterBibliography} 
\setlength{\itemsep}{1pt} 
 }
\end{document}